\documentclass[12pt]{article}

\usepackage[hyperfootnotes=false,colorlinks=true]{hyperref}
\usepackage{amsthm,amssymb,amscd,amsmath,amstext,mathrsfs,amsfonts,pb-diagram,enumerate,graphicx}

\usepackage[usenames,dvipsnames]{color}
\usepackage{ifthen}
\usepackage{makeidx}

\makeindex
\newcommand{\rev}[2]{\textcolor{blue}{#1}\index{#2}\ensuremath{^{\text{\textcolor{red}{#2}}}}}
\newcommand{\secrev}[2]{\textcolor{Brown}{#1}}
\newcommand{\thirdrev}[2]{\textcolor{Orange}{#1}\footnote{#2}}

\renewcommand{\rev}[2]{#1}
\renewcommand{\secrev}[2]{#1}
\renewcommand{\thirdrev}[2]{#1}{}

\newcommand{\binomial}[2]{\left(\begin{matrix}#1\\#2\end{matrix}\right)}

\newtheorem{theorem}{Theorem}
\theoremstyle{plain}

\newtheorem{corollary}[theorem]{Corollary}
\newtheorem{lemma}[theorem]{Lemma}
\newtheorem{proposition}[theorem]{Proposition}
\theoremstyle{definition}
\newtheorem{definition}[theorem]{Definition}
\newtheorem{remark}[theorem]{Remark}
\numberwithin{equation}{section}

\def\al{\alpha}

\def\ga{\gamma}

\def\de{\delta}
\def\De{\Delta}
\def\ep{\varepsilon}
\def\ze{\zeta}

\def\ka{\kappa}
\def\la{\lambda}

\def\si{\sigma}

\def\Si{\Sigma}

\def\Om{\Omega}

\def\C{{\mathbb C}}
\def\G{{\mathbb G}}
\def\K{{\mathbb K}}
\def\L{{\mathbb L}}
\def\N{{\mathbb N}}
\def\P{{\mathbb P}}

\def\R{{\mathbb R}}

\def\U{{\mathbb U}}

\def\CD{\mathcal{D}}
\def\CG{\mathcal{G}}

\def\CI{\mathcal{I}}
\def\CL{\mathcal{L}}
\def\CM{\mathcal{M}}
\def\CN{\mathcal{N}}
\def\CP{\mathcal{P}}

\def\CR{\mathcal{R}}
\def\CS{\mathcal{S}}

\def\CW{\mathcal{W}}
\def\CX{\mathcal{X}}
\def\CY{\mathcal{Y}}
\def\CZ{\mathcal{Z}}

\def\GLS{{\mathbb G }{\mathbb L}_{n,m}^>}
\def\Knm{{\mathbb K}^{n \times m}}

\newcommand{\im}{\mbox{Im }}
\newcommand{\rank}{\mbox{rank }}
\newcommand{\trace}{\mbox{\rm trace }}
\newcommand{\diag}{\mbox{diag }}

\newcommand{\hess}[1]{{D^2#1}}
\newcommand{\hessk}[1]{{D^2_{\ka}#1}}

\def \grad {{\rm grad\ }}

\newcommand{\ra}{\rightarrow}

\renewcommand{\Re}{\mathrm{Re}}

\def\x{\dot{x}}

\def\A{{ A}}

\def\CL{\mathcal{L}}
\def\CP{\mathcal{P}}
\def\GL{\mathbb{GL}_{n,m}}
\def\Pk{\mathcal{P}_{(k)}}
\def\Dk{\mathcal{D}_{(k)}}
\def\SD{\overline{\mathcal{SD}^{\,2}}}

\def\Knm{\K^{n \times m}}

\def\RN{\mathrm{N}}

\newcommand{\immersion}{i}

\title{Convexity properties of the condition number II
\thanks{Mathematics Subject Classification (MSC2010): 53C23 (Primary), 65F35, 15A12 (Secondary).}
}
\author{Carlos Beltr\'an
        \thanks{C. Beltr\'an,
               Departamento de Matem\'aticas, Estad. y Comput.
               Universidad de Cantabria,
               Santander, Espa\~na               
               ({\tt carlos.beltran@unican.es}). CB was supported by MTM2007-62799 \secrev{and MTM2010-16051, Spanish Government}{C1}. 
}              
\and      Jean-Pierre Dedieu
          \thanks{J.-P. Dedieu,
               Institut de Math\'ematiques,
               Universit\'e Paul Sabatier,
               31062 Toulouse cedex 09, France
                ({\tt jean-pierre.dedieu@math.univ-toulouse.fr}). J.-P. Dedieu was 
supported by the ANR Gecko and by the Fields Institute at Toronto.
}
\and    Gregorio Malajovich
        \thanks{G. Malajovich,
                Departamento de Matem\'atica Aplicada,
		Instituto de Matem\'atica,
                Universidade Federal do Rio de Janeiro,
                Caixa Postal 68530,
                CEP 21945-970, Rio de Janeiro, RJ, Brazil
                ({\tt gregorio.malajovich@gmail.com}). He was partially supported by 
CNPq, FAPERJ and CAPES from Brazil, and by
the Brazil-France agreement of cooperation in Mathematics.
}
\and    Mike Shub
				\thanks{M. Shub,
\secrev{CONICET, IMAS, Universidad de Buenos Aires, Argentina and CUNY Graduate School, New York, NY, USA.
({\tt shub.michael@gmail.com})
M.S. was partially 
supported by a NSERC Discovery Grant, by CONICET PIP 0801 2010-2012 and by
ANPCyT PICT 2010-00681.
}{}}
\thanks{\secrev{J.-P.D., G.M. and M.S were partially supported by the
MathAmSud grant {\em Complexity}.}{}}
}
\begin{document}
\addtocounter{footnote}{1}
\maketitle
\medskip

\centerline {This paper is dedicated to Steve Smale, on his 80th birthday.}

\begin{abstract}
\rev{In our previous paper~\cite{BDMS}, we studied the condition metric
in the space of \thirdrev{maximal rank}{Instead of non singular} $n \times m$ matrices. Here, we show that
this condition metric induces a Lipschitz-Riemann structure on \thirdrev{that space}{Avoid repetition of ``on the space of XXX matrices''}.}{mi01}
After investigating
geodesics in such a nonsmooth structure, we show that
the inverse of the smallest singular value of a matrix is
a log-convex function \secrev{along geodesics}{JP1} (Theorem~\ref{th-1}). 

We also show that a similar result holds for the solution
variety of linear systems (Theorem~\ref{the-7.8}). 

Some of our intermediate results, such as Theorem~\ref{th-2},
on the \rev{second covariant derivative or }{ma10} {\em Hessian} 
of a function with symmetries on a manifold,
and Theorem~\ref{the-5.2} on piecewise self-convex functions, 
are of independent interest.

\rev{Those results were motivated by our investigations on the
complexity of path-following algorithms for solving polynomial
systems.}{S2}
\end{abstract}

\section{Introduction}

Let two integers $1 \leq n \leq m$ be given and let us consider the space of matrices $\Knm$, $\K = \R$ or $\C$,
equipped with the Frobenius inner product
$$\left\langle M,N \right\rangle_F = \trace (N^*M) = \sum_{i,j}m_{ij}\overline{n_{ij}}.$$
We denote by
$$\si_1(A) \ge \ldots \ge \si_{n-1}(A) \ge \si_n(A) \ge 0$$
the singular values of a matrix $A \in \Knm$, by $\GL$ the space of matrices $A \in \Knm$ with maximal rank, that is
$\rank A = n$ or, equivalently, $\si_n(A) > 0$, and by $\CN$ the set of singular (or rank deficient) matrices: 
$$\CN = \Knm \setminus \GL = \left\{ A \in \Knm \ : \ \si_n(A) = 0 \right\} .$$
The distance of a matrix $A \in \Knm$ from $\CN$ is given by its smallest singular value:
$$d_F(A, \CN) = \min_{S \in \CN} \left\| A-S \right\|_F = \si_n(A).$$

Consider now the problem of connecting two matrices with the shortest possible path staying, as much as possible, away
from the set of singular matrices. We realize this objective by considering an absolutely continuous path $A(t)$, $a \le
t \le b$, with given endpoints (say $A(a) = A$ and $A(b) = B$) which minimizes 
its {\bf condition length} defined by
$$L_\ka = \int_a^b \left\|\frac{dA(t)}{dt}\right\|_F \si_n(A(t))^{-1}dt.$$
We call {\bf minimizing condition path} an absolutely continuous path which minimizes this integral in the set of
absolutely continuous paths with the same end-points. We define a {\bf minimizing condition geodesic} as a minimizing
condition path parametrized by the condition arc length, that is when
$$ \left\|\frac{dA(t)}{dt}\right\|_F \si_n(A(t))^{-1} = 1 \mbox{ a. e.}$$
A {\bf condition geodesic} is an absolutely continuous path which is locally a minimizing condition geodesic.
This concept of geodesic is related to the Riemannian structure defined on $\GL$ by:
$$\left\langle M,N \right\rangle_{\ka, A} = \si_n(A)^{-2} {\Re}\left\langle M,N \right\rangle_F.$$
We call it the {\bf condition Riemann structure} on $\GL$. 

Our objective is to investigate the properties of the smallest singular value $\sigma_n(A(t))$ along a condition
geodesic. Our main result says: 

\begin{theorem} \label{th-1} For any condition geodesic $t \ra A(t)$ in $\GL$, the map 
\rev{$t \ra \log \left(\si_n^{-2}
(A(t))\right)$}{ma01} is convex. 
\end{theorem}

This theorem extends our main result in \cite{BDMS}. In that paper, the same theorem is proven for those condition
geodesic arcs contained in the open subset
$$\GLS = \left\{ A \in \GL \ : \ \sigma_{n-1}(A) > \sigma_{n}(A)\right\}$$
that is when the smallest singular value $\sigma_{n}(A)$ is simple. The reason for this restriction is easy to explain.
The smallest singular value $\sigma_n(A)$ is smooth in $\GLS$, and, in that case, we can use the toolbox of Riemannian
geometry. But it is only locally Lipschitz in $\GL$; for this reason we call the condition structure in $\GL$ a {\bf
Lipschitz-Riemannian structure}. 

%


\bigskip
\noindent
\rev{{\bf Motivation:}}{S1}
Let us now say a word about our motivations. 
The today classical papers \cite{bez1}, \cite{bez2}, and \cite{bez5} by
Shub and Smale relate complexity bounds for homotopy methods to solve 
\thirdrev{polynomial systems}{Remark 1}
 to the condition number of the
encountered problems along the considered homotopy path. Ill-conditioned problems slow the algorithm and increase its
complexity. For this reason it is natural to consider paths which avoid ill-posed problems, and, at the same time, are
as short as possible. The condition metric has been designed to construct such paths. It has been introduced by Shub in
\cite{shu}, then studied by Beltr\'an and Shub in \cite{bel} in spaces of polynomial equations 
\thirdrev{(see also~\cite{DMS, Beltran, Malajovich}.) When we started
to work in this project, we expected to reduce the more general problem of
finding good homotopy paths for non-linear systems to the `linear' case.
Unfortunately this seems to be a harder problem, to be pursued later.}{Reintroduced references + a transition sentence.} 

The case of linear maps (and related spaces) appears in Beltr\'an-Dedieu-Malajovich-Shub \cite{BDMS} and Boito-Dedieu \cite{boi}. 
\thirdrev{}{REMOVED TEXT: The main result relating condition metric and complexity appears in
~\cite{DMS}(see also~\cite{Beltran} and ~\cite{Malajovich}): there is an algorithm 
that, given a continuous path 
$(f_t)_{t \in [0,1]}$ in the
space of systems of homogeneous polynomial equations, computes
a mesh $0=t_0 < t_1 < \cdots < t_k=1$ and approximate zeros
(in the sense of Smale) $x_i$ associated to $\zeta_{t_i}$,
where
$f_t(\zeta_t) \equiv 0$. This algorithm terminates in time
linear in
\[
D^{3/2} \epsilon^{-2} L(f_t, \zeta_t)
\]
where $D$ is the maximum degree of the equations, $\epsilon$ is
an arbitrary parameter and $L(f_t, \zeta_t)$ is the 
{\bf condition length} of the path $(f_t, \zeta_t)$. 
}

In the linear case, it is rather a remarkable fact \secrev{that}{C2} the inverse of the \rev{squared}{ma01} distance to singular matrices 
\rev{$\si_n^{-2} (A(t))$}{ma01}
is log-convex along the condition geodesics. So, in particular, the maximum of \rev{$\log \left( \sigma_n(A(t))^{-2}\right)$}{ma01}
and the maximum of $ \sigma_n(A(t))^{-1}$ along such paths is necessarily obtained at its endpoints 
and the condition geodesics stay away from singular matrices. 

\thirdrev{This is clearly not true in the usual metric, since
straight lines can get arbitrarily close to the variety of degenerate
matrices. This suggests the following application:}{Self-convexity}

\thirdrev{
If we consider a condition path connecting a given $A \in \GL$
to (for example) $I_{n,m}\| A \|_F / \sqrt{n}$ ($I_{n,m}(i,j) = 1$ if $i=j$ and $0$ otherwise), for any matrix $A(t)$ in this path, 
according to Theorem \ref{th-1}, one has
$$\frac{\sqrt{n}}{\| A \|_F} \leq \sigma_n(A(t))^{-1} \leq \sigma_n(A)^{-1}.$$
We think this property may help to find good \thirdrev{preconditioners}{Remark 4} to solve linear systems. 
}{Self-convexity}

\thirdrev{}{REMOVED TEXT(JP):
We think this property may help to find good preconditioners
to solve linear systems which would be a good area for future
research.} 

\thirdrev{}{REMOVED TEXT(JP):
We don't know if the analogue of Theorem~\ref{the-7.8} holds 
in the solution variety corresponding to the space of polynomial systems,
although Proposition~\ref{restricted-case} proves it for systems which
vanish at a given point. If self-convexity of the condition number where
to hold into the 
\thirdrev{solution}{Remark 5} variety itself we would have a good geometric
picture of what is involved in choosing a homotopy path in an optimal
(or near optimal) way. 
}

\thirdrev{
There are other motivations. Convexity of the distance or similar function to the ill-posed
problems may play a role in optimization. Witness for example the role
played by the barrier function in linear programming theory. Two of us
will be expanding on this theme in a forthcoming
paper.}{Self-convexity}

\bigskip

\noindent
\rev{{\bf Outline of the paper}}{S1}

\rev{The condition number is not of class $\mathcal C^1$,
hence we cannot apply the usual Riemannian geometry to the
condition metric. In {\bf Section 2}, we introduce Lipschitz-
Riemann structures and develop the basic results, that allow
us to do differential geometry in the non-smooth case.
Using nonsmooth analysis techniques, we prove that any
condition
geodesic is $\mathcal C^1$ with a locally Lipschitz derivative (Theorem \ref{the-2.1}). Such techniques are already present in
Boito-Dedieu \cite{boi}.}{S1}

\rev{In {\bf Section 3} we develop an important tool for proving self-convexity,
allowing a more systematic use of the symmetries. (A symmetry is an
isometry of a manifold that leaves a function invariant).
Theorem \ref{th-2} gives a simplified computation of the Hessian when there is a Lie group of symmetries. This theorem
may be of independent interest. It is so natural we would not be surprised if it is already known, but we have not found
it anywhere. We were led to this theorem sometime after a conversation with John Lott on Hessians and Riemannian
submersions while he was visiting the University of Toronto.}{S1} 

\rev{The strategy for proving the main theorem is to decompose the space
of matrices in a finite union of smooth manifolds, so that in
each of them the metric is smooth. In {\bf Section 4} we produce
this decomposition, we study the group of symmetries of the condition
number and then, using Theorem~\ref{th-2}, we establish self-convexity
on each piece.}{S1}

\rev{In {\bf Section 5}, we prove a result that 
may be of independent interest, Theorem \ref{the-5.2}: piecing
together convexity results on restrictions of the Lipschitz-Riemann structure to a union of submanifolds of varying
dimensions, where the structure is smooth, to obtain a global result. 
}{S1}

\rev{In {\bf Section 6}, we use all these tools to finish the proof of 
Theorem~\ref{th-1}.
We use the same tools in {\bf Section 7} to state and prove
Theorem~\ref{the-7.8} about self-convexity in the solution variety}{S1}
\thirdrev{
\[
\CW = \left\{
(A,x) \in \mathbb{GL}_{n,n+1} \times \P(\K^n): Ax = 0
\right\} \ .
\]
\rev{Above, the notation $\mathbb P(\mathbb E)$ denotes the projectivization
of a linear space $\mathbb E$. Namely, it is the space (manifold) of real or 
complex
lines in $\mathbb E$ passing through the origin. For instance,
$\mathbb P(\mathbb R^3)$ is the classical projective plane, that can
also be obtained by identifying antipodal points of the sphere $\mathbb{S}^2$.}{mi02}
}{Misplaced text: this was in Motivations while it belongs to the outline.}

\bigskip

\noindent
{\bf Acknowledgements.} 
We have benefited greatly from conversations with our colleagues Charles Pugh and Vitaly Kapovitch about
Lipschitz-Riemann structures, especially those conformally equivalent to smooth structures by locally Lipschitz scaling
maps.

Some of this work was accomplished when we met not only in our intitutions, but also at the
Institut de Matem\`{a}tica de la Universitat de Barcelona \secrev{and}{M3} at the Thematic Program
in the Foundations of Computational Mathematics (FoCM) at the Fields Institute. We thank
these institutions. \secrev{Also, we would like to thank an anonymous
referee \thirdrev{for many helpful comments}{Added text}.}{}

\section{Geodesics in Lipschitz-Riemann structures, and self-convexity} \label{sec-2}

\subsection{Lipschitz-Riemann structures} \label{sec-2.1}

\rev{Most textbooks of Riemannian geometry define a Riemannian structure
on a smooth manifold $\CM$ as a scalar product
$\langle \cdot , \cdot \rangle_x$ on each tangent space $T_x \CM$, depending smoothly on
$x$. Here we drop the smoothness hypothesis.
\begin{definition}
A {\bf  Lipschitz-Riemann structure} on a $\mathcal C^2$ manifold $\CM$ is
a scalar product
$\langle \cdot , \cdot \rangle_x$ 
at each $T_x \CM$, such that its coefficients are locally Lipschitz functions 
of $x$. Also,
let $\| u \|_x = \sqrt{ \langle u,u \rangle_x}$ be the associated norm in
$T_x\CM$.
\end{definition}
}{S1} 

The {\bf length} of an absolutely continuous path $x(t) \in \CM$, $a \le t \le b$, is defined as the integral
$$\rev{L(x,a,b) = \int_a^b \| \dot x(\tau) \|_{x(\tau)} d\tau,}{mi03}$$
where $\dot x(t)$ denotes the derivative with respect to $t$. Its  {\bf arc length} is given by the map 
$$t \in [a,b] \ra L(x,a,t) \in [0, L(x,a,b)].$$
The {\bf distance} $d(a,b)$ between two points $a,b \in \CM$ is the infimum of all the lengths of the paths containing
$a$ and $b$ in their image.
We call {\bf minimizing  path} an absolutely continuous path such that $L(x,a,b) = d(a,b)$. 
\medskip

\rev{It is usual in differential geometry textbooks to construct
geodesics as solutions of a certain second order differential equation,
the {\em geodesic differential equation}. Unfortunately, the coefficients
of this equation are given by a formula in terms of the partial
derivatives of the metric coefficients. In a Lipschitz-Riemann structure,
those coefficients are assumed to be Lipschitz, not necessarily
differentiable functions. Also, it turns out that minimizing paths are
not necessarily smooth.}{ma02}

We define a {\bf minimizing 
geodesic} 
as a minimizing  path parametrized by  arc length, that is when
$$ \left\| \dot x(t) \right\|_{x(t)} = 1 \mbox{ a. e.}$$
A path in $\CM$ parametrized by arc length is a {\bf geodesic} when it is locally a minimizing geodesic.

The main result of this section is the following: 

\begin{theorem} \label{the-2.1} Any geodesic \rev{for a Lipschitz-Riemann 
structure}{ma05} belongs to the class $\mathcal C^{1+ Lip}$ that is $\mathcal C^1$ with a locally Lipschitz
derivative. 
\end{theorem}

This theorem is proved in \thirdrev{section \ref{sec-2.5}}{Remark 7}, it extends a similar result by Charles Pugh
\cite{pug} who proves the existence of locally minimizing $\mathcal C^{1+ Lip}$ geodesics. 
His argument is based on a smooth approximation of the Lipschitz structure where the classical toolbox of Riemannian
geometry applies, 
followed by a {\it passage \`a la limite}.

Using different techniques we prove here this regularity assumption for all geodesics. 

\rev{An immediate consequence of ~\cite[Cor .VIII-4 p.126]{Brezis} is
that 
$\mathcal C^{1+Lip} = W^{2, \infty}$ the Sobolev space of maps $f$ with $f'' \in L^\infty.$}{mi04}

\subsection{Existence of geodesics in a Lipschitz-Riemann structure}\label{sec-2.2}

Existence of minimizing geodesics with given endpoints may be deduced from the Hopf-Rinow Theorem. \rev{Because we cannot assume the smoothness of
geodesics, we refer to Gromov's version of this theorem
\cite[Th.1.10]{gro}.}{ma02}
A metric space $(X,d)$ is a {\bf path metric space} if the distance between each pair of points equals the infimum of
the lengths of curves 
joining the points.

\begin{theorem} \label{th-2-2} If $(X,d)$ is a complete, locally compact path metric space, then
\begin{itemize}
\item Each bounded, closed subset is compact,
\item Each pair of points can be joined by a minimizing geodesic.
\end{itemize}
\end{theorem}

Two examples of such spaces are given by Boito-Dedieu \cite{boi} for linear maps ($X$ is one of the 
connected components of $\GL$ equipped with the condition structure), and by Shub \cite{shu} when $X$ is
the solution variety associated with the homogeneous polynomial system solving problem equipped with the corresponding
condition structure. 
\subsection{Lipschitz-Riemann structures in $\R^k$, generalized gradients
and the problem of Bolza} \label{sec-2.3}
\label{sec-2.4}
\rev{}{ma04}

An important example of Lipschitz-Riemann structure is given by an open set $\Om \subset \R^k$ equipped with the scalar
product
$$\langle u , v \rangle_x = v^T H(x) u$$
where $H$ is a locally Lipschitz map \thirdrev{from $\Om$ into the set of positive definite $n \times n$ matrices.}{Remark 8} 

A minimizing geodesic  $x(t) \in \Om$, $a \le t \le b$, minimizes the integral 
$$\int_a^b\sqrt{\dot y(t)^T H(y(t)) \dot y(t)} dt$$ 
in the set of absolutely continuous paths $y(t)$ with endpoints $y(a) = x(a)$, and $y(b) = x(b)$. This is an instance of
the Bolza problem. 

For a smooth integrand $L$, a local solution $x(t)$ of the Bolza problem 
\rev{
\begin{equation}\label{bolza-problem}
\inf \int_a^b L(y(t),\dot y(t)) dt,
\end{equation}
}{mi07}
where the infimum is taken in the set of a.c. paths with given endpoints, satisfies the Euler-Lagrange differential
equation
\thirdrev{
\begin{equation}\label{euler-lagrange}
- \frac{d}{dt}\frac{\partial L}{\partial \dot x}(x(t),\dot x(t)) + \frac{\partial L}{\partial  x}(x(t),\dot x(t)) = 0
\ \ a.e.\end{equation}}{Remark 9}
In our context, \rev{it is possible to differentiate $L(x, \dot x) = \sqrt{\dot x^T H(x) \dot x}$ with respect to the second argument by ordinary
differential calculus:
\[
\frac{\partial}{\partial \dot x} 
L(x, \dot x): y \mapsto 
\frac{1}{L(x, \dot x)} \dot x^T H(x) y
.
\]
If we avoid $\dot x = 0$ (which will be the case), 
we deduce that $L$ 
is smooth in the variable $\dot x$
and locally Lipschitz in the variable $x$.}{mi05} 
For this reason
\rev{we replace the classical geodesic differential equation 
by a generalized
version of the Euler-Lagrange equation \thirdrev{\eqref{euler-lagrange}}{Remark 8} based on generalized gradients.}{ma02}

Let $f : \Om \subset \R^k \ra \R$ be a locally Lipschitz function defined on an open set. Its {\bf one-sided directional
derivative} at $x \in \Om$ in the direction $d \in \R^k$ is defined as
$$f'(x,d) = \lim_{t \ra 0_+} \frac{f(x+td)-f(x)}{t}.$$
The {\bf generalized directional derivative in Clarke's sense} of $f$ at $x \in \Om$ in the direction $d$ is defined as
$$f^o(x,d) = \limsup_{
{\scriptsize \begin{array}[pos]{l}
	y \ra x\\
	t \ra 0_+
\end{array}
}} \frac{f(y+td)-f(y)}{t}$$
and the {\bf generalized gradient} of $f$ at $x$ is the nonempty compact subset of $\R^k$ given by
$$\partial f(x) = \left\{s\in \R^k \ : \ \left\langle s,d \right\rangle \leq f^o(x,d) \ {\rm for} \ {\rm all} \ d\in
\R^k \right\}.$$
\rev{It turns out that the generalized gradient is always a convex set.}{mi06}
When $f \in \mathcal C^1(\Om)$ the generalized gradient is just the usual one: 
\secrev{$\partial f(x) = \left\{\nabla f(x)\right\}.$}{JP4,C11}
The generalized directional derivative is related to the gradient via the equality
$$
f^o(x,d) = \max_{s \in \partial f(x)} \left\langle s,d \right\rangle .
$$
We say that $f$ is {\bf regular at }$x$ when the two directional derivatives exist and are equal:
$$f^o(x,d) = f'(x,d) \mbox{ for any } d \in \R^k.$$

When $f$ is defined on a $\mathcal C^1$ manifold $\CM$, we say that $f$ is {\bf regular at }$m \in \CM$ when its composition with
a local chart at $m$
gives a regular map in the usual meaning. 

Good references for this topic is Clarke \cite{cla} or Schirotzek \cite{sch}.

For the problem of Bolza described above
the counterpart of the Euler-Lagrange 
equation is given by the following result (see \cite{cla} Theorem 4.4.3, and \cite{cla-a}).

\begin{theorem} \label{th-2-3} Let $x$ solve the Bolza problem
\rev{\eqref{bolza-problem}}{mi07} 
in the case in which $L(x,\dot x)$ is a locally Lipschitz
map 
and suppose that $\dot x$ is essentially bounded. Then there is an absolutely continuous map $p$ such that
$$\dot p(t) \in \partial_x L(x(t), \dot x(t)) \mbox{ and } p(t) \in \partial_{\dot x} L(x(t), \dot x(t)) \ \ a. e.$$
\end{theorem}

\subsection{Proof of Theorem \ref{the-2.1}}\label{sec-2.5}
\label{sec-2.6}
\rev{Since Theorem~\ref{the-2.1} is of local nature, it suffices to prove it 
locally in $\R^k$. Once this is done, take a local chart and transfer the
Lipschitz-Riemann structure of $\CM$ 
to an open set $\Om \subset \R^k$ where the theorem is already proved.
Therefore, let us show the theorem in $\R^k$.}{ma04}

By definition, a geodesic is a locally minimizing geodesic.
Thus, it suffices to establish the theorem in this case.

A minimizing geodesic  $x(t) \in \Om$, $a \le t \le b$, is parametrized by arc length so that
$$\dot x(t)^T H(x(t)) \dot x(t) = 1 \ \ a.e.,$$
Thus, $\dot x(t)$ is $\ne 0$ and essentially bounded:
$$\dot x \in L^\infty \left( [a,b], \R^k \right).$$
Moreover, $x$ minimizes the integral 
$$\int_a^b \sqrt{\dot y(t)^T H(y(t)) \dot y(t)} dt$$ 
in the set of absolutely continuous paths with endpoints $y(a) = x(a)$, and $y(b) = x(b)$. Thus, according to Theorem
\ref{th-2-3},
there is an absolutely continuous arc $p$ such that
\rev{
\begin{eqnarray}
\label{eq-1-of-page-8}
\dot p(t) &\in& \partial_x \sqrt{\dot x(t)^T H(x(t)) \dot x(t)}, \\
\label{eq-2-of-page-8}
p(t) &\in& \partial_{\dot x} \sqrt{ \dot x(t)^T H(x(t)) \dot x(t)}
\end{eqnarray}}{mi08}
for almost all $t \in [a,b]$. Since our integrand is smooth in the $\dot x$ variable we may write \rev{\eqref{eq-2-of-page-8}}{mi08}
$$p(t) = \frac{H(x(t)) \dot x(t)}{\sqrt{ \dot x(t)^T H(x(t)) \dot x(t)}} = H(x(t)) \dot x(t).$$
Thus, $\dot x(t) = H(x(t))^{-1}p(t)$ is absolutely continuous and $x(t)$ possesses a.e. a second derivative $\ddot x(t)
\in L^1 ([a,b], \R^k).$

We now have to show that this second derivative is essentially bounded. 
This comes from \rev{\eqref{eq-1-of-page-8}}{mi08}.
Since $\sqrt{\cdot}$ is a smooth function we get
\secrev{from Proposition 2.3.3 and Theorem 2.3.9 of Clarke's book
~\cite{cla} that}{C12}
$$\partial_x \sqrt{ \dot x^T H(x) \dot x } \subset 
\frac{ \dot x^T \partial H(x) \dot x}{2\sqrt{\dot x^T H(x) \dot x}} = \frac{1}{2} \dot x^T \partial H(x) \dot x,$$
with
$$\dot x^T \partial H(x) \dot x = \sum_{i,j} \dot x_i \dot x_j \partial h_{ij}(x).$$
\rev{Equation \eqref{eq-1-of-page-8}}{mi08}
implies
$$\dot p(t) \in \frac{1}{2} \dot x(t)^T \partial H(x(t)) \dot x(t)\text{ a.e.}$$
From the hypothesis, the functions $h_{ij}(x)$ are locally Lipschitz. Their generalized gradients are compact convex
sets 
in $\R^k$. The union of all these sets along the path $x(t)$ gives us a bounded set. Since the curve $\dot x(t)$ is
continuous, we
deduce from these considerations, that $\dot p(t)$ is bounded a.e. Thus $p(t)$ is Lipschitz,
and $\dot x(t)
=H(x(t))^{-1}p(t)$ is also Lipschitz. The second derivative $\ddot x(t)$ is thus bounded by the Lipschitz constant of
$\dot x(t)$,
and we are done.

\begin{remark} \label{rem-2-1} The previous lines give the following properties for a geodesic $x$ in $\Om$: 
$x \in \mathcal C^{1+Lip}$, $\dot x^T H(x) \dot x = 1$, and
$$\frac{d}{dt}(H(x) \dot x) \in  \frac{1}{2}\dot x^T \partial H(x) \dot x = \frac{1}{2}\sum_{i,j} \dot x_i \dot x_j
\partial h_{ij}(x)\text{ a.e.}$$
The initial value problem, and even the boundary value problem associated with this second order differential inclusion,
may have many solutions. 
Examples are given in \cite{boi}.
Moreover, solutions are not necessarily locally minimizing geodesics and
geodesics are not necessarily unique. 
\end{remark}

\subsection{Conformal Lipschitz-Riemann structure}\label{sec-2.7}

The example of a \rev{Lipschitz-Riemann}{mi09} structure which motivates this paper is given by the condition structure on $\GL$. 
It is obtained in multiplying the Frobenius scalar product by the locally Lipschitz function $\si_n^{-2}$.
Let us put it in a more general setting. 

\rev{
\begin{definition}\label{def-3-MODIFIED}
Let\/ $(\CM,\langle\cdot,\cdot\rangle)$ be 
a $\mathcal C^2$ Riemannian manifold, and let
$\al : \CM \rightarrow \R$ be a locally Lipschitz function with
positive values. Let $\CM_\ka$ be the manifold $\CM$ with the new
metric\[ \langle\cdot,\cdot\rangle_{\ka , x} = \al
(x)\langle\cdot,\cdot\rangle_x
\]
\thirdrev{called $\alpha$-Riemann structure. When $\alpha$ is the square of the (unscaled) condition number, i.e. $\alpha(A)=\|A^{\dagger}\|_2^2=\sigma_n^{-2}$, this is also called the {\bf condition Riemann structure} 
or simply the {\bf condition structure}.}{Remark 10}
We say that $\al$ is {\bf self-convex}
when\/ $\log \al(\ga (t))$ is convex for any geodesic $\ga$ in~$\CM_\ka$.
\end{definition}
}{mi10}

We denote by $L$ (respectively
$L_\ka$) the length of a curve $\ga$ in the $\CM$-structure (respectively in the $\CM_\ka$-structure). We will speak of
{\bf length} or {\bf condition length}, and also of {\bf distance} or 
{\bf condition distance}, {\bf geodesics} or {\bf condition geodesics} and so on.

Examples of self-convex maps are given in \cite{BDMS} where this concept is introduced for the first time.

Using this definition Theorem \ref{th-1} above reads 
$$\alpha(A)=\sigma_n(A)^{-2} \mbox{ is self-convex in } \GL .$$
\section{ Self-convexity in the smooth case and the computation of Hessians.}

\subsection{Self-convexity in the smooth case}

\rev{Self convexity in the smooth case was studied in our previous
paper~\cite{BDMS} in this journal. We
refer the reader to Section 2 
of~\cite{BDMS} for basic definitions regarding convexity and geodesic
convexity.
A snapshot of the main features of self-convexity in the smooth case
follows.}{ma06}
\bigskip

\rev{We denote by $D$ the Levi-Civita connection and by
 $D_XT$ the {\bf covariant derivative} of a tensor $T$
in the direction given by a vector field $X$. Recall that if we assume geodesic
coordinates in the neighborhood of a point $p$, then $(D_XT)_p$ is
the same as the ordinary directional (or Lie) derivative.
The covariant derivative is coordinate independent, in the sense
that $D_XT$ is a tensor.}{S1}

\rev{If $f$ is a \thirdrev{smooth enough}{Remark 11} function, then its derivative with respect
to a vector field is denoted by $X(f)$, so that $X(f)(p) = 
Df(p) X(p) = \langle \nabla f(p), X(p) \rangle_p$.
The second covariant derivative of a function $f$ (sometimes
also known as the {\em Hessian}) is defined by
\begin{equation}\label{intr-hessian}
\hess{f} (X,Y) = 
D(Df)(X,Y) = 
D_X(Y(f)) -(D_XY)(f)
\end{equation}
where $X$ and $Y$ are
smooth vector fields. The operator above is symmetric, in 
the sense that $\hess{f} (X,Y) = \hess{f} (Y,X)$ 
(see e.g. ~\cite[p.305]{Berger})}{mi18}

When $\alpha : \CM \rightarrow \R$ is $\mathcal C^2$, self-convexity of $\alpha$ is equivalent to the \rev{second covariant derivative
of $\log(\alpha)$}{mi12} being positive semi-definite in the $\alpha$-condition Riemann structure (see \cite{udr} Chap. 3,
Theorem 6.2). Note that the second covariant derivative of a map $\CM \ra \R$ is different in $\CM$ and in $\CM_\ka$. We
denote them respectively by \rev{$\hess{}$ or $\hessk{}$}{mi12}.
Self-convexity of $\alpha$ is equivalent to 
\rev{$\hessk{\log(\alpha)}$}{mi12} being positive semi-definite. 

\rev{Proposition~2 of~\cite{BDMS} is 
\begin{proposition}
For a function $\alpha: \CM \rightarrow \mathbb R$ of class $\mathcal C^2$ with
positive values self-convexity is equivalent to
\begin{equation} \label{eq:Hesssmooth}
2\alpha(x)
\rev{\hess{\alpha(x)}}{ma10}
(\dot x,\dot x) + \|D\alpha(x)\|_x^2  \|\dot x\|_x^2 - 4(D\alpha(x)\dot x)^2 \geq 0
\end{equation}
for \thirdrev{all}{Remark 12} $x\in \CM$ and for \thirdrev{all}{Remark 12} vector $\x \in T_x\CM$, the tangent space at $x$.
\end{proposition}}{ma06}

\subsection{Self-convexity in a product space}

\rev{Proposition~2 of~\cite{BDMS}}{mi11}
has an immediate corollary which can be useful. Suppose $\CN$ is another
\secrev{$\mathcal C^2$}{C15} Riemannian manifold. Give $\CM \times \CN$ the product 
\rev{metric}{mi13}. Let $\pi: 
\CM \times \CN \rightarrow \CM$ be the projection on the first factor and
$\hat \alpha: \CM \times \CN \rightarrow \mathbb R$ be the composition
$\hat \alpha = \alpha \circ \pi$.

\begin{proposition}\label{product}
\rev{Let $\alpha$ be of class $\mathcal C^2$ in $\CM$. Then,
$\alpha$ is self-convex in $\CM$ if and only if $\hat \alpha$ is 
self-convex in $\CM \times \CN$.}{mi14}
\end{proposition}

\rev{We thank an 
anonymous referee for pointing out the
{\em if} part of this Proposition and simplifying the proof.}{mi14-15}

\begin{proof}
\rev{We prove first the only if part.}{mi14}
Let $(x, y) \in \CM \times \CN$ and assume normal (geodesic) coordinates
in a neighborhood of $x \in \CM$. Also, assume normal coordinates around 
$y \in \CN$ with
respect to the inner product 
$\langle \cdot , \cdot \rangle_{\CN}$. 

We claim that this defines a system of normal coordinates in
$\CM \times \CN$. 
\rev{This can be seen from the fact that the exponential map
in a product manifold $\CM \times \CN$ is the partitionning of the exponential
mappings of $\CM$ and $\CN$. However, we give a direct proof below.}{mi15}

\secrev{Let $g_{ij}$ and $\Gamma_{ij}^k$ denote respectively the
coefficients of the first fundamental form $\langle \cdot , \cdot \rangle_{x,y}$
and the Christoffel symbols.}{C16}
By construction, 
$g_{ij}(x,y) = \delta_{ij}$. Also, it is easy to see
that for all indexes $i,j,k$,
\[
\Gamma_{ij}^k (x,y) = 0 \ .
\]

Indeed, if indices $(i,j,k)$ correspond to the same 
component $\CM$ or $\CN$ this follows from the choice of
normal coordinates in each component. Otherwise, say that
$i, j$ correspond to coordinates in  $\CM$ and $k$ to
coordinates in  $\CN$. Then $g_{ik} \equiv g_{jk} \equiv 0$
and furthermore,
\[
\frac{\partial}{\partial u^k} g_{ij} (x,y) = 0 .
\]
Thus $\Gamma_{ikj} (x,y) = 0$ for all indexes $i,j,k$. This implies
that $\Gamma_{ij}^k (x,y) = 0$ as well. Thus we have a normal system
of coordinates around $(x, y) \in \CM \times \CN$.
\medskip

In that system of coordinates,
\rev{
\[
D^2_{\CM \times \CN} \hat \al(x,y) =
\left[
\begin{matrix}
D^2_{\CN} \al(x) & 0 \\
0 & 0 
\end{matrix}
\right]
\]
}{mi16}
\rev{From the block structure of the second covariant derivative above,
it is clear that 
$D^2_{\CM \times \CN} \hat \al(x,y)$ is
\secrev{}{C17} positive \thirdrev{semi-definite}{Remark 13} if and only if 
$D^2_{\CN} \al(x)$ is positive definite.}{mi15}
%

\end{proof}

We have raised the question in the introduction of whether self-convexity
of the condition number holds for the condition Riemann structure on the 
solution variety considered in~\cite{shu}. The theorems proven in this 
paper apply to the case of linear \secrev{systems}{C18}, but with the use of Proposition~\ref{product}
they give us some information on polynomial systems almost for free.\thirdrev{}{REMOVED: the proposition follows from Propositions 19,9 and Theorem 29. Moved and changed in the proof of Proposition 11}

\medskip

Let $\mathbf d = (d_1, \dots, d_n)$. Consider the vector space
\[
\CP_{\mathbf d, 0} = \{ (f_1, \dots, f_n): f_i \in \mathbb C[x_1, \dots, x_n]
\mbox{ with } \deg f_i = d_i \mbox{ and } f_i(0)=0 \}
.
\]

\secrev{An important point is that self-convexity is well-defined
for Riemannian manifolds. Therefore, \thirdrev{if}{Remark 14} we want to speak of
self-convexity in $\CP_{\mathbf d,0}$, we need to make it into
an inner product vector space.
We will follow~\cite{BCSS} and assume the unitarily invariant metric
in the space of degree $d_i$ polynomials. This is the same as the 
metric for symmetric $d_i$-tensors. Then we define the product
metric for $\CP_{\mathbf d}$ and it is inherited by the subspace
$\CP_{\mathbf d, 0}$. 
In more precise terms: if $f_i(x) = \sum_{1 \le |a| \le d_i} f_{ia} 
x_1^{a_1} x_2^{a_2} \cdots x_n^{a_n}$ and
$g_i(x) = \sum_{1 \le |a| \le d_i} g_{ia} 
x_1^{a_1} x_2^{a_2} \cdots x_n^{a_n}$ then we set
\[
\langle f, g \rangle =
\sum_{i=1}^n
\sum_{1 \le |a| \le d_i}
\frac{ f_{ia} \bar g_{ia} }{\binomial{d_i}{a}}
\]
with
\[
\binomial{d_i}{a} = \frac{d_i!}{a_1! a_2! \dots a_n! (d_i-|a|)!}
.
\]
This vector space splits as $\CP_{\mathbf d, 0} = \CL_0 \oplus \mathrm{(H.O.T.)}_0$ where
$\CL_0$ are linear and $\mathrm{(H.O.T.)}_0$ are higher order polynomials vanishing
at $0$. Those two spaces are orthogonal. 
The inner product for linear polynomials is 
\[
\begin{split}
\langle A x, B x \rangle &= \sum_{i=1}^n \frac{1}{d_i}\sum_{j=1}^n
A_{ij} \bar B_{ij} 
=\\
&=\mathrm{tr} \left( 
\left(
\left[
\begin{matrix}
1/\sqrt{d_1} & & \\
& \ddots & \\
& & 1/\sqrt{d_n}
\end{matrix}
\right] 
B \right)^*
\left(
\left[
\begin{matrix}
1/\sqrt{d_1} & & \\
& \ddots & \\
& & 1/\sqrt{d_n}
\end{matrix}
\right] 
A \right)
\right)
.
\end{split}
\]
The unscaled\thirdrev{\cite[Prop.5 p.228]{BCSS}}{Remark 16}, normalized\thirdrev{\cite[p.233]{BCSS}}{Remark 16} condition number is defined, 
for $f \in \CP_{\mathbf d,0}$,
by
\[
\mu(f,0) =  
\left\| Df(0)^{-1}
\left[
\begin{matrix}
\sqrt{d_1} & & \\
& \ddots & \\
& & \sqrt{d_n}
\end{matrix}
\right] 
\right\|_2
=
\sigma_n^{-1} \left( 
\left[
\begin{matrix}
1/\sqrt{d_1} & & \\
& \ddots & \\
& & 1/\sqrt{d_n}
\end{matrix}
\right] 
Df(0) \right).
\]
\thirdrev{The right-hand term is the (unscaled) condition number for
$\mathcal L_0$. It coincides with the unscaled condition number for
matrices, which is the topic of this paper.}{Remark 15}
\begin{proposition}\label{restricted-case}
$\mu$ is 
self-convex in its domain of definition $\CP_{\mathbf d, 0} \setminus
\Sigma$, where $\Sigma = \{ \mathbf f \in \CP_{\mathbf d,0}: D\mathbf f(0)
\text{ is degenerate}\}$.
\end{proposition}
}{M8 C19}

\begin{proof}
\thirdrev{Immediate from Proposition \ref{product} and Theorem~\ref{th-1}.}{ADDED TEXT}
\end{proof}
\medskip

\subsection{Computation of the Hessian}

When analyzing the convexity properties of $\sigma_n(A)$, we first note that this function \rev{is invariant through unitary changes of coordinates,
namely}{ma08}
$$\sigma_n(A)=\sigma_n(UAV^*)$$ 
for unitary matrices $U \in \U_n$, and $V \in \U_m$ (resp.
orthogonal matrices $U \in \mathbb O_n$, and $V \in \mathbb O_m$). 
Let us consider this situation in a general framework. 

\rev{A {\em Lie group} is a group that is also a smooth manifold,
and such that the group operations (multiplication and inversion)
are smooth. We say that a Lie group $G$ acts (smoothly) on 
a manifold $\CM$ if there is a smooth map $\ell: G \times \CM \rightarrow \CM$
with
\[
\ell( (g_1 g_2) , p) = \ell( g_1 , \ell(g_2, p))
\text{\ and \ }
\ell( 1, p ) = p \ .
\]
In the example above, $G=\U_n \times \U_m$ acts on
$\GL$ by $\ell( (U,V),p) = U p V^*$. 
For simplicity, we may write $g(p)$ for $\ell(g,p)$ and
assimilate $g$ to the mapping $p \rightarrow g(p)= \ell(g,p)$.

\thirdrev{An {\bf isometry} of $\CM$ is a diffeomorphism of $\CM$ that
preserves Riemannian distance.}{Remark 18}
We say that the Lie group $G$ acts by isometries when for all $g$,
the corresponding map $g:p \rightarrow g(p)$ is an isometry of $\CM$.

\begin{definition} Let $\alpha: \CM \rightarrow \mathbb R$. 
A {\bf group of symmetries} of $\alpha$ is a Lie group, acting
smoothly by isometries on $\CM$, and leaving $\alpha$ invariant (that is,
$\alpha (g(p)) = \alpha(p)$ for all $g \in G$ and $p \in \CM$.
\end{definition}
\thirdrev{Let $1$ be the unit of the group $G$.
We will denote by $\frak g$ the Lie algebra of $G$ and
by $\exp: \frak g \simeq T_1G \rightarrow G$ the exponential function
(See, for instance, \cite{Kirillov}). For instance, if 
$G=\mathbb U_n$, then $1$ is the $n \times n$ identity matrix, and
$\frak g$ is the algebra of skew-Hermitian matrices. Moreover,
$\exp(A) = I + A + \frac{1}{2}A^2 + \frac{1}{3!}A^3 + \cdots$.}{Remark 17}

Note that it may happen (for instance, if $G$ is a discrete group)
that $\frak g=\{0\}$ and hence $T_pG_p=\{0\}$.
}{ma08}

\rev{Given $p \in $ \thirdrev{$\CM$}{Remark 19}, $G(p) = \{ g(p): g \in G\}$ will denote the
$G$-orbit of $p$. The orbit $G(p)$ is a manifold \thirdrev{\cite[Cor 2.19]{Kirillov}}{Remark 20}. If the group $G$
is compact, the orbit is then \thirdrev{an embedded submanifold of $\CM$}{ADDED TEXT: instead of simply an embedded manifold.}. In any case,
$T_pG(p)$ will denote the tangent space of the orbit $G(p)$ at
$p$, as a subspace of $T_p\CM$. It can also be described
as the set of all
\[\frac{d}{dt}(\exp(ta)(p))\left.\right|_{t=0},
\]
for $a \in \frak g$, the Lie algebra of $G$.}{ma08}

\rev{For instance, when $G=\U_n \times \U_m$, then $\frak g$ is
$\A_n \times \A_m$ (the \thirdrev{skew-Hermitian}{Remark 21} matrices) and
$\exp(ta)$ is the usual matrix exponential: 
\[
\exp(t(a_1, a_2))(p) = 
\left(
I + t a_1 + \frac{t^2}{2} a_1^2 + \cdots 
\right)
p
\left(
I + t a_2 + \frac{t^2}{2} a_2^2 + \cdots 
\right)^*
\]
}{ma08}

\begin{theorem}\label{th-2}
\rev{Let $\CM$ be a smooth Riemannian manifold. 
Let $\alpha: \CM \rightarrow \mathbb R$ be of class $\mathcal C^2$,
and let $G$ be a group of symmetries of $\alpha$. Let $p \in \CM$.}{ma08}
Let $w=b+k\in T_p\CM$ where $k\in T_pG(p)$, $b\perp T_pG(p)$.
Let the vector field $K$ be the infinitesimal generator associated with some element $a$ in
the Lie Algebra $\mathfrak{g}$ of $G$, where $k=\frac{d}{dt} (\exp(ta)(p))\left.\right|_{t=0}$. Namely,
\[
K(q)=\frac{d}{dt}(exp(ta)q)\left.\right|_{t=0},\;\;q\in \CM.
\]
Let $\phi_t(q)=\phi(t,q)$ be the flow of $\grad\alpha$, defined for $t\in(-\varepsilon,\varepsilon)$ and $q$ close
enough to $p$. Let $B$ be a smooth vector field  in $\CM$ such that $B(\phi_t(p))=D\phi_t(p)b$ \rev{where $D$ denotes the usual  
derivative applied to the 
diffeormorphism $\phi_t: \CM \rightarrow \CM$ 
}{ma08}. Then, the following
equality holds: 
$$\hess{\alpha(p)}(w,w)=$$
\rev{\[
\hess{\alpha(p)}(b,b)+\frac{1}{2}\langle \grad(\|K\|^2)(p),\grad\alpha(p)\rangle_p+
\grad\alpha(\langle B,K\rangle)(p).
\]}{ma10}
\end{theorem}
\rev{Above, 
$\grad\alpha(\langle B,K\rangle)(p) =
\langle \grad \alpha(p), \grad(\langle B,K \rangle_p) \rangle_p$
is the directional derivative of 
$\langle B,K\rangle$ with respect to $\grad \alpha$.}{ma08}

Let us recall \rev{from \eqref{intr-hessian} the intrinsic definition of 
the second covariant
derivative or {\em Hessian}.}{mi18}
\rev{\[
\hess{\alpha(p)}(v,w)=X(Y(\alpha))_p-(D_XY)(\alpha)_p,
\]}{ma10}
where $X,Y$ are vector fields, $X(p)=v$, $Y(p)=w$, and $D$ is the Levi-Civita connection. 
\rev{
Also, $[X,Y]$ is the {\bf Lie bracket} of two vector fields
$X$ and $Y$. It is defined for any $\alpha$ of class $\mathcal C^2$ by
\[
[X,Y] (\alpha) = X(Y(\alpha))-Y(X(\alpha)) .
\]
It turns out that this is a first order differential operator,
hence $[X,Y]$ is a vector field.}{mi17}

\rev{Another useful identity relating the Lie bracket and the
Levi-Civita connection is:
\begin{equation}\label{eq-bracket-connection} 
[X,Y] = D_X Y - D_Y X
\end{equation}
}{mi19}

The proof of Theorem
\ref{th-2} is a consequence of the two following lemmas:

\begin{lemma}\label{lem-1}
For any vector field $X$ on $M$, we have
\[
2\hess{\alpha}(X,K)=
\grad\alpha(\langle X,K\rangle)-\langle [\grad\alpha,X],K\rangle.
\]
Moreover,
\rev{
\begin{equation}\label{eq:hessian1}
\hess{\alpha(p)}(k,k)=\frac{1}{2}\langle\grad(\|K\|^2)(p),\grad\alpha(p)\rangle_p,
\end{equation}}{ma10}
\end{lemma}

\begin{proof}
We recall that for vector fields $X,Y,Z$,
\begin{equation}\label{eq:connectionformula}
2\langle D_XY,Z\rangle=X(\langle Y,Z\rangle)+Y(\langle X,Z\rangle)-Z(\langle X,Y\rangle)+
\end{equation}
\[
\langle [X,Y],Z\rangle+\langle [Z,X],Y\rangle-\langle [Y,Z],X\rangle.
\]
Note that $K(p)=k$ and $K(q)\in T_q G(q)$ for $q\in \CM$. As $\alpha$ is $G$-invariant,
\begin{equation}\label{eq:10}
K(\alpha)=\langle K,\grad \alpha\rangle=0.
\end{equation} Moreover, the one-parameter group generated by $K$ consists of global isometries, 
thus $K$ is a Killing vector field, which implies that for any pair of vector fields $X,Y$,
\[
\langle D_YK,X\rangle+\langle D_XK,Y\rangle=0,\text{ or equivalently
using (\ref{eq:connectionformula})}
\]
\begin{equation}\label{eq:13}
K(\langle Y,X\rangle)+\langle [Y,K],X\rangle+\langle [X,K],Y\rangle=0.
\end{equation}
We can now compute

\[
2\hess{\alpha}(X,K)=2X(K(\alpha))-2(D_XK)(\alpha)=-2\langle D_XK,\grad\alpha\rangle=
\]
\[
-X(\langle K,\grad\alpha\rangle)-K(\langle X,\grad\alpha\rangle)+\grad\alpha(\langle X,K\rangle)
\]
\[
-\langle [X,K],\grad\alpha\rangle-\langle [\grad\alpha,X],K\rangle+\langle [K,\grad\alpha],X\rangle.
\]
From (\ref{eq:13}) we know that 
$$-K(\langle X,\grad\alpha\rangle)-\langle [X,K],\grad\alpha\rangle+\langle [K,\grad\alpha],X\rangle=0.$$ 
Using  $\langle\grad\alpha,K\rangle=0$, we conclude
\[
2\hess{\alpha}(X,K)=
\grad\alpha(\langle X,K\rangle)-\langle [\grad\alpha,X],K\rangle,
\]
which proves the first assertion. 

When $X = K$, the second term above vanishes:
\rev{using \eqref{eq-bracket-connection},}{mi19}
\begin{eqnarray*}
\langle[K,\grad \alpha],K\rangle&=&
\langle D_K\grad \alpha,K\rangle-\langle D_{\grad \alpha}K,K\rangle
\\
&=&
\langle D_K\grad \alpha,K\rangle+\langle \grad \alpha,D_{K}K\rangle\\
&=&K(\langle K,\grad \alpha\rangle)\\
&=&0.
\end{eqnarray*}
Equation (\ref{eq:hessian1}) follows.

\end{proof}

\begin{lemma}\label{lem-2}
\rev{\[
2\hess{\alpha(p)}(k,b)=\grad\alpha(\langle B,K\rangle)(p).
\]}{ma10}
\end{lemma}

\begin{proof}
By continuity of the formulas in the lemma, we can assume that $k\neq0$ and that $b,\grad\alpha(p)$ are lineary independent.
%
%
Let $\RN_0$ be a codimension $2$ submanifold of $\CM$ with $p$ in its interior. Assume that 
$b\in T_p\RN_0$, $k$ is orthogonal to $T_p \RN_0$, and $\grad\alpha(p)\not\in T_pN_0$.

Let $\RN = \cup \phi_t(\RN_0)$ with $\phi_t$ the flow associated with $\grad\alpha$ and where the union is 
taken in a small interval around $t=0$. $\RN$ is a  codimension $1$ submanifold. For small $\varepsilon$, 
the integral curve of $\grad \alpha$ is thus contained in $\RN$, and for $q=\phi_t(p)$, 
we have $B(q)=D\phi_t(p)b\in T_q\RN$. Both $\grad\alpha$ and $B$ are tangent to $\RN$ by construction. 
By Frobenius Theorem, $[B,\grad \alpha]$ is again tangent to $\RN$. In particular, $[\grad \alpha,B](p)\in T_p \RN$, 
and hence $\langle [\grad \alpha,B],K\rangle(p)=0$. From Lemma \ref{lem-1},
\[
2\hess{\alpha}(B,K)=\grad\alpha(\langle B,K\rangle)-\langle [\grad\alpha,B],K\rangle=\grad\alpha(\langle B,K\rangle)
\]
at $p$ as wanted.
\end{proof}

\vskip 3mm
\begin{proof}[Proof of Theorem~\ref{th-2}] 
The \rev{second covariant derivative}{ma07} is a symmetric bilinear form. Thus,
\[
\hess{\alpha(p)}(v,v)=\hess{\alpha(p)}(b,b)+\hess{\alpha(p)}(k,k)+2\hess{\alpha(p)}(b,k).
\]
Theorem \ref{th-2} follows from lemmas \ref{lem-1} and \ref{lem-2}. 
\end{proof}

\begin{corollary}\label{cor-1}
 Assume that for every $p\in \CM$:
\begin{itemize}
\item \rev{$\hessk{\log(\alpha)(p)}$}{ma10} is positive semi-definite in $\left(T_pG(p)\right)^\perp$,
\item For $b\in T_p\CM$, $b\perp T_pG(p)$, we have that $D\phi_t(p)b\perp T_{\phi_t(p)}G(\phi_t(p))$. 
Here, $\phi_t(q)=\phi(t,q)$ is the flow of $\grad\alpha$, defined for $t\in(-\varepsilon,\varepsilon)$ 
and $q$ close enough to $p$.
\item For every $a\in\mathfrak{g}$, the associated vector field 
$K(q)=\frac{d}{dt}(exp(ta)q)\left.\right|_{t=0}$, $q\in M$, satisfies
\[
\alpha D(\|K\|^2)(\grad\alpha)+\|K\|^2\|\grad\alpha\|^2\geq0.
\]
\end{itemize}
Then, $\alpha$ is self-convex in $\CM$.
\end{corollary}

\begin{proof}
$\al$ is self-convex if and only if  $\hessk{\log(\alpha)}$ 
is positive semi-definite. Now, let $v=b+k\in \CM$. According to Theorem \ref{th-2},
\[
\hessk{\log(\al)(p)}(v,v)=\hessk{\log(\al)(p)}(b,b)+
\]
\rev{\[
\frac{1}{2}\langle {\rm grad}_\kappa ((\|K\|_\ka)^2)(p),{\rm grad}_\kappa \log(\al)(p)\rangle_{\ka,p} +
{\rm grad}_{\kappa,p} \alpha(\langle B,K\rangle_\ka)(p),
\]}{ma10}
where $K$ is as defined in Theorem \ref{th-2} and $B$ is a vector field such that $B(\phi_t(p))=D\phi_t(p)b$. 
Note that ${\rm grad}_\kappa \alpha(\langle B,K\rangle_\ka)$ depends only on the value of $B$, 
and $K$ along the integral curve $\phi_t(p)$. Moreover,
\[
\langle B,K\rangle_\ka(\phi_t(p))=\al(\phi_t(p))\langle B(\phi_t(p)),K(\phi_t(p))\rangle= 
\]
\[
\al(\phi_t(p))\langle D\phi_t(p)(b),K(\phi_t(p))\rangle=0,
\]
from the second item in the hypotheses of our corollary. Thus, we have
$$
\hessk{\log(\al)(p)}(v,v)=
$$
\rev{\[
\hessk{\log(\al)(p)}(b,b)+
\frac{1}{2}\langle {\rm grad}_\kappa ((\|K\|_\ka)^2)(p),{\rm grad}_\kappa \log(\al)(p)\rangle_{\ka,p}.
\]}{ma10}
This quantity has to be non-negative for every $v$ or, equivalently,
\begin{itemize}
\item \rev{$\hessk{\log\al(p)}$}{ma10}  has to be positive semi-definite in $\left(T_pG(p)\right)^\perp$, and
\item \rev{$\langle {\rm grad}_\kappa ((\|K\|_\ka)^2)(p),{\rm grad}_\kappa \log(\al)(p)\rangle_{\ka,p}\geq0$}{ma10}
for every vector field $K$, $K(q)=\frac{d}{dt}(exp(ta)q)\|_{t=0}$ where  $a\in \mathfrak{g}$.
\end{itemize}
The second of these two items can be re-written using the original Riemannian structure $\langle\cdot,\cdot\rangle$. 
Note that
\[
(\|K\|_\ka)^2=\al\|K\|^2,
\]
\[
{\rm grad}_\kappa ((\|K\|_\ka)^2)=\frac{1}{\al}\grad(\al\|K\|^2)= \grad\|K\|^2+\|K\|^2\frac{\grad\alpha}{\alpha},
\]
\[
{\rm grad}_\kappa \log(\al)=\frac{1}{\al}\grad(\log\al)=\frac{\grad\alpha}{\alpha^2}.
\]
Thus,
\begin{eqnarray*}
\langle \grad_\kappa ((\|K\|_\ka)^2),{\rm grad}_\kappa \log(\al)\rangle_\ka
&=&
\langle \grad\|K\|^2+\|K\|^2\frac{\grad\alpha}{\alpha},\frac{\grad\alpha}{\alpha^2}\rangle
\\
&=&
\frac{1}{\al^3}\left(\alpha D(\|K\|^2)(\grad\al)+\|K\|^2\|\grad\al\|^2\right).\\
\end{eqnarray*}
The corollary follows.
\end{proof}

\section{Self-convexity in spaces of matrices}\label{sec-4}

Let $u \leq n$ and $(k)=(k_1, \ldots , k_u) \in \N^u $ such that $k_1 + \cdots + k_u=n$. We define $\CP_{(k)}$ as the
set of matrices $A \in \GL$ 
with $u$ distinct singular values
$$\si_1(A) > \cdots > \si_u(A) > 0,$$
$\si_i(A)$ having the multiplicity $k_i$. Such a matrix has a singular value decomposition $A=UDV^*$ with $U \in \U_n$,
$V \in \U_m$ and $D \in \GL$ with
$$D = \diag \left(\overset{k_1}{\overbrace{\sigma_{1},\ldots,\sigma_{1}}},\ldots,
\overset{k_u}{\overbrace{\sigma_{u},\ldots,\sigma_{u}}}\right) = 
\diag \left(\sigma_{1}I_{k_1}, \ldots , \sigma_{u}I_{k_u}\right). $$
\rev{Above, $\mathbb U_n$ is the group of unitary $n\times n$ matrices.
If $\K = \R$, it should be replaced by the group of orthogonal
$n \times n$ matrices.}{mi21}

We also let 
$$\CD_{(k)}=\left\{ D \in \CP_{(k)} \ : \ D = \diag \left(\sigma_{1}I_{k_1}, \ldots , \sigma_{u}I_{k_u}\right), \ \si_1
> \dots > \si_u \right\}.$$

\rev{Notice that the singular values $\sigma_1 > \cdots > \sigma_u$
can vary within each $\CP_{(k)}$ or each $\CD{(k)}$.
}{mi20}

\begin{proposition}\label{prop-1}
$\CP_{(k)}$ is a real smooth embedded submanifold of $\GL$. Its real codimension is
\begin{itemize}
\item $k_1^2+\cdots+k_u^2-u$ if $\K=\C$.
\item $\frac{1}{2}(n+k_1^2+\cdots+k_u^2)-u$ if $\K=\R$.
\end{itemize}
The tangent space to $\CP_{(k)}$ at a matrix
\[ D = 
\begin{pmatrix}
\sigma_1I_{k_1}&0&0&0& \cdots &0\\
0&\ddots&0&0& \cdots &0\\
0&0&\sigma_u I_{k_u}&0& \cdots &0
\end{pmatrix}
\]
is the set of matrices 
\[
\begin{pmatrix}
\lambda_1I_{k_1}+{A}_1&*&*&*& \cdots &*\\
*&\ddots&*&*& \cdots &*\\ 
*&*&\lambda_uI_{k_u}+{A}_u&*& \cdots &*
\end{pmatrix}
\]
where ${A}_1, \ldots, {A}_u$ are skew-symmetric matrices of respective sizes $k_1, \ldots, k_u$, $\lambda_1, \ldots,
\lambda_u \in \R$, and the other entries are complex numbers (real, if $\K=\R$). 
Moreover, for any $i = 1, \ldots, u$, $\si_i : \CP_{(k)} \ra  \R$ is a smooth function.
\end{proposition}

\begin{proof} To prove that $\CP_{(k)}$ is a real smooth embedded submanifold of $\GL$ we use Lemma \ref{lem-7} (see the
appendix). We take
$G = \U_n \times \U_m$, $\CM = \GL$, and $\CD = \CD_{(k)}$. The group action of $G$ on $\CM$ is given by
$$(U,V,X) \in G \times \GL \ra UXV^* \in \GL.$$
Under this action, the image of $\CD_{(k)}$ is $\CP_{(k)}$. \thirdrev{Define the equivalence relation 
$\CR$ in $\mathbb U_n \times \mathbb U_m  \times \CD_{(k)}$ by
\[
(U,V,D) \CR (U',V',D') \text{ if and only if } UDV^* = U'D'V'^*
.
\]
Since $D$ is diagonal}{Remark 22} this is equivalent to
$$D'=D, \ \ U'=UM, \ \ V'=VM_W,$$
where $M$ and $M_W$ are unitary block-diagonal matrices
$$M = \diag ( U_1, \ldots , U_u), \ \ M_W = \diag ( U_1, \ldots , U_u, W),$$
with $U_i \in \U_{k_i}$, and $W \in \U_{m-n}$. Note that the set $\CI_{(k)}$ of such pairs $(M,M_W)$ is the isotropy
group of any \thirdrev{$D \in \CD_{(k)}$.}{Remark 23}
Also, the relation $\CR$ is invariant under left 
\rev{$\U_n \times \U_m$}{mi22} action, namely:
\[
(U,V,D) \CR (U',V',D') \Leftrightarrow 
(QU,RV,D) \CR (QU',RV',D') 
\]
for any \thirdrev{$(Q,R) \in \U_n \times \U_m$}{Remark 24}.

It is easy to see that the graph of this equivalence relation, that is the set of pairs $((U,V,D), (UM,VM_W,D))$, with
$U$, $V$, $D$, $M$, 
and $W$ as before, is a closed submanifold in $(G \times \CD_{(k)}) \times (G \times \CD_{(k)})$. 
\thirdrev{Indeed, this graph is the image of the diffeomorphic embedding
\[
\begin{matrix}
G \times \CD_{(k)} \times \mathbb U_{(k)}
\times \mathbb( U_{(k)}\times\mathbb{U}_{m-n})&\rightarrow&(G \times \CD_{(k)}) \times (G \times \CD_{(k)})\\
((U,V),D,M,M_W)&\mapsto&((U,V,D),(UM,VM_W,D))\end{matrix}
\]
($\mathbb U_{(k)} = \mathbb U_{k_1} \otimes \cdots \otimes
\mathbb U_{k_u}$ are the unitary block-diagonal matrices).}{Remark 25}

Thus the quotient
space 
$(G \times \CD_{(k)}) / \CR$ is equipped with a unique manifold structure making $\pi$ (the canonical surjection) a
submersion. 

Let us define 
$$\immersion : (G \times \CD_{(k)}) / \CR \ra \GL, \ \ \immersion (\pi (U,V,D)) = UDV^*.$$
\rev{The injectivity of $\immersion$ follows by construction of $\CR$: 
elements of 
$(G \times \CD_{(k)})/ \CR$ are represented non-uniquely by elements
$(U,V,D) \in (G \times \CD_{(k)})$. Two of those 
elements (say $(U,V,D)$ and $(U',V',D')$) represent the same
equivalence class if and only if $UDV^* = U'D'V'^*$.}{mi25-26}

We \rev{still}{mi25} have to check that this map is an immersion. For any $(\dot U, \dot V, \dot D)$ in the tangent space $T_{(U,V,D)} G
\times \CD_{(k)}$ we have 
$$D (\immersion \circ \pi) (U,V,D)(\dot U, \dot V, \dot D) = \dot U D V^* + U \dot D V^* + U D \dot V^* = U(AD + \dot D
- DB)V^*$$
with $\dot U = UA$, $\dot V = VB$, $A$ and $B$ skew-symmetric matrices of respective size $n$ and $m$. When $AD + \dot D
- DB = 0$, we obtain,
via an easy computation, 
$$\dot D = 0, \ \ A = \diag (A_1, \ldots , A_u), \ \ B = \diag (A_1, \ldots , A_u, C),$$
where $A_i$ and $C$ are skew-symmetric matrices of respective sizes $k_i$ and $m-n$. Thus $(\dot U, \dot V, \dot D) =
(UA, VB, 0)$ is tangent to the fiber
of $\pi$ in $G \times \CD_{(k)}$ above $\pi (U, V, D)$ so that $D\pi(U,V,D)(\dot U, \dot V, \dot D) = 0$. In other words
$$D \immersion (\pi (U,V,D))(D\pi(U,V,D)(\dot U, \dot V, \dot D)) = 0 \Longrightarrow D\pi(U,V,D)(\dot U, \dot V, \dot
D) = 0$$
that is $D \immersion (\pi (U,V,D))$ is injective. 

The last point to check to apply Lemma \ref{lem-7} is the continuity of the inverse of $\immersion$. Suppose that $X_p
\ra X$ with $X_p, X \in \im \immersion = \CP_{(k)}$. 
We can write them $X_p = U_p D_p V_p^*$ and $X = U D V^*$. Let $(U_{p_q}, V_{p_q})$ be a subsequence which converges to
$(\tilde U, \tilde V)$ ($G$ is compact). 
Since $X_{p_q} \ra X$ we have $D_{p_q} \ra \tilde U^* X \tilde V = \tilde D,$ and $\tilde U \tilde D \tilde V^* = U D
V^*$. 
Now we consider the sequence $\tilde U^* X_p \tilde V$. It is a
convergent sequence, hence it has a unique limit $\tilde D$
and $(\tilde U, \tilde V, \tilde D) \CR (U,V,D)$.
Thus, $\pi(\tilde U^* U_p, \tilde V^* V_p, D_p)$
converges to $\pi(I,I,D)$. By left 
$\U_n \times U_m$ action, we conclude that
$\pi(U_p, V_p, D_p)$ converges to $\pi (U, V, D)$ as required.

Thus, the hypothesis of Lemma \ref{lem-7} is satisfied and $\CP_{(k)}$ is a real smooth embedded submanifold of $\GL$. 

The computation of its dimension is easy:
it is given by the difference of the dimension of $G \times \CD_{(k)}$ and the dimension of the fiber above any point in
the quotient space, 
that is 
$$\dim \U_{n} + \dim \U_{m} + u - \dim \U_{k_1} - \ldots - \dim \U_{k_u} - \dim \U_{m-n}.$$

The tangent space $T_D \CP_{(k)}$, $D = \diag(\si_1 I_{k_1}, \ldots , \si_u I_{k_u})$, is the image of the tangent space
 $T_{(I_n,I_m,D)} G \times \CD_{(k)}$
by the derivative $D(\immersion \circ \pi)(I_n,I_m,D)$. It is the set of matrices $AD + \dot D - DB$ with $\dot D =
\diag (\la_1 I_{k_1}, \ldots , \la_u I_{k_u})$, $A$ and $B$ skew symmetric of sizes $n$ and $m$.
They all have the type
described in Proposition \ref{prop-1} and this space of matrices has the right dimension. 

Let us prove the smoothness of the map $X \in \CP_{(k)} \ra \si_i(X) \in \R$. Since the map 
$(U,V,D) \in G \times \CD_{(k)} \ra \si_i(D)$ is smooth, and constant in the equivalence classes, the map
$\pi(U,V,D) \in (G \times \CD_{(k)})/\CR \ra \si_i(D)=\si_i(UDV^*)$ is also smooth. Thus the map 
$X=UDV^* \in \CP_{(k)} \ra \si_i(X)$ is smooth as the composition of the previous map by $\immersion^{-1}$.
\end{proof}

\rev{
\begin{lemma}\label{lem-svd} Let $I$ be an open interval.
Let $(\gamma(t))_{t \in I}$ be a smooth path in $\CP_{(k)}$. Then,
there are smooth paths $U(t) \in \U_n$, $V(t)\in U_m$ and $\Sigma(t) \in
\CD_{(k)}$ so that
\begin{equation}\label{eqsvd}
\gamma(t) = U(t) \Sigma(t) V(t)^* 
\end{equation}
for all $t \in I$.
\end{lemma}}{ma11}
\thirdrev{We give two quite different proofs of this result.}{ADDED TEXT: Alternative proof of the same result}
\begin{proof}
\thirdrev{Consider the mapping $\pi:\mathbb{U}_n\times\mathcal{D}_{(k)}\times\mathbb{U}_m\rightarrow\mathcal{P}_{(k)}$ sending $(U,D,V)$ to $UDV^*$. Note that $\pi$ is surjective. We claim that it is also a submersion: by unitary invariance, we may assume that $U=I_n$, $V=I_m$. Then, for skew-symmetric matrices $A,B$ of respective sizes $n,m$, we have:
\[
D\pi(I,D,V)(A,\dot D,B)=AD+\dot D+DB.
\]
It is a simple exercise to check that one can get any matrix in the tangent space $T_D\mathcal{P}_{(k)}$, computed in Proposition \ref{prop-1}, by choosing appropriate $A,B$. Thus, $D\pi(I,D,I)$ is surjective, and $\pi$ is a submersion.}{ADDED TEXT}

\thirdrev{Finally, we claim that $\pi$ is also a proper map (i.e. the preimage of a compact set is a compact set): let $K\subseteq\mathcal{P}_{(k)}$ be a compact subset. The mapping sending a matrix to its (ordered) singular values is continuous, and hence the set of singular values of matrices in $K$ is the continuous image of a compact set, thus a compact set, call it $K'\subseteq\mathcal{D}_{(k)}$. Thus, $\pi^{-1}(K)$ is a closed (because $\pi$ is continuous) subset of the compact set $\mathbb{U}_n\times K'\times\mathbb{U}_m$, thus a compact set. This proves that $\pi$ is proper.}{ADDED TEXT}

\thirdrev{A theorem by Ehresmann \cite{Ehresmann1951} (see \cite[Th. 5.1]{Rabier1997} for a general version on a more modern framework) says that, under these hypotheses, $\pi$ is actually a locally trivial fibration which implies that it defines a fiberbundle. Hence, $\pi$ has the homotopy lifting property and in particular any path in $\mathcal{P}_{(k)}$ can be smoothly lifted to a path in $\mathcal{U}_n\times\mathcal{D}_{(k)}\times\mathcal{U}_m$ as wanted.}{ADDED TEXT}

\end{proof}
\thirdrev{As an alternative, we have:}{ADDED TEXT}
\begin{proof}
\rev{
We will show that $U(t)$, $V(t)$ and $\Sigma(t)$ are solutions
of a certain differential equation on the manifold 
$\U_n \times \U_m \times \CD_{(k)}$.
An important fact to be used below is that $T_I\U_n$ is
the space of skew-hermitian matrices. In the real case,
$T_I\mathbb O_n$ is the space of skew-symmetric matrices. 
Let us assume for a while that \eqref{eqsvd} admits a solution.
Differentiating \eqref{eqsvd}
with respect to $t$, we obtain after a few trivial manipulations that
\[
U(t)^* \dot \gamma(t) V(t) = 
U(t)^* \dot U(t) \Sigma(t)
- \Sigma(t) V(t)^* \dot V(t) 
+ \dot \Sigma(t)
.
\]
For shortness, let $M(t) = 
U(t)^* \dot \gamma(t) V(t)$, $A(t)= U(t)^* \dot U(t) \in T_I\U_n$
and $B(t) = V(t)^* \dot V(t) \in T_I\U_m$. We have now:
\[
M(t) = A(t) \Sigma(t) - \Sigma(t) B(t) + \dot \Sigma(t) .
\]
Using block notation, we obtain for $i<j$ that
\[
M_{ij}(t) 
= \sigma_j(t) A_{ij}(t) - \sigma_i(t) B_{ij}(t) .
\]
The equation for block $M_{ji}(t)$ reads:
\[
M_{ji}(t) = \sigma_i(t) A_{ji}(t) - \sigma_j(t) B_{ji}(t) .
\]
Transposing,
\[
M_{ji}(t)^* = -\sigma_i(t) A_{ij}(t) + \sigma_j(t) B_{ij}(t) .
\]
We obtain therefore
\begin{equation}\label{svd-non-diag}
\left\{
\begin{array}{lcl}
A_{ij} &=& \frac{ 1}{\sigma_j^2 - \sigma_i^2} 
\left(\sigma_j M_{ij}(t) + \sigma_i M_{ji}(t)^* \right)
\\
B_{ij} &=& \frac{ 1}{\sigma_j^2 - \sigma_i^2} 
\left(\sigma_i M_{ij}(t) + \sigma_j M_{ji}(t)^* \right)
\end{array}
\right.
\end{equation}
The blocks in the diagonal (that is, $i=j$) are of the form
\[
M_{ii}(t)  = \sigma_i (A_{ii} - B_{ii}) + \dot \sigma_i I_{k_i},
\]
hence we can solve by setting
\begin{equation}\label{svd-diag}
A_{ii} = -B_{ii} = \frac{1}{2 \sigma_i} (M_{ii}(t) - \dot \sigma_i I_{k_i}) .
\end{equation}
Equations \eqref{svd-non-diag}-\eqref{svd-diag} are a system of
smooth non-autonomous ordinary differential equations
in variables $U \in \U_n, V \in \U_m$ and $\Sigma \in \CD_{(k)}$.
The Lipschitz condition holds.
Hence, for every $t_0 \in I$, 
there are $\epsilon > 0$ and local solutions $U(t)$, $V(t)$ and
$\Sigma(t)$ for $t \in (t_0-\epsilon, t_0+\epsilon)$, solving
\eqref{eqsvd}.
}{ma11}

\rev{In order to show the existence of a global solution on all
the interval, we need to check that as $t \rightarrow t_0+\epsilon$,
the solution converges to
a limit in $\U_n \times \U_m \times \CD_{(k)}$. The convergence
of $U(t)$ and $V(t)$ follows from compactness of the unitary
group.
Because $\gamma(t_0+\epsilon) \in \CP_{(k)}$,
\[
\lim_{t \rightarrow t_0 + \epsilon} \Sigma(t) \in \CD_{(k)} . 
\] 
Hence, the solution $(U(t), V(t), \Sigma(t))$ can be extended
to an interval that is open and closed in $I$, hence to all $I$.
}{ma11}
\end{proof}

\rev{Let $\alpha:\GL$ be defined by
$\alpha(A) = \sigma_n(A)^{-2}$. We also denote by 
$\al=\sigma_u^{-2}$ its restriction to $\CP_{(k)}$ or to
$\CD_{(k)}$.
We \thirdrev{first}{Remark 26} consider the case of diagonal matrices, then
we prove self-convexity of $\alpha$ in $\CP_{(k)}$.}{mi23}

\begin{lemma}\label{lem-4}
Let $\CP_{(k)}$ be equipped with the condition metric structure 
$$\left\langle \cdot , \cdot \right\rangle_\ka = \si_u^{-2} {\Re}\left\langle \cdot , \cdot \right\rangle_F.$$
\begin{enumerate}
\item If $\Sigma_1,\Sigma_2\in\Dk $, then any minimizing condition geodesic in $\Pk$ joining $\Sigma_1$ and $\Sigma_2$
lies in $\Dk $, 
\item The set $\Dk $ is 
\rev{a totally geodesic submanifold of $\CP_{(k)}$ for the condition metric}{ma11}, namely, every
geodesic in $\Dk $ for the induced structure is also a geodesic in $\CP_{(k)}$, or equivalently:
\item If $\Sigma \in \Dk $ and $\dot{\Sigma}\in T_\Sigma\Dk $, then the unique geodesic in $\CP_{(k)}$ through $\Sigma$
with tangent vector 
$\dot{\Sigma}$ at $\Sigma$, remains in $\Dk $.
\end{enumerate}
Moreover, $\al = \sigma_u^{-2}$ is log-convex in $\Dk $.
\end{lemma}

\begin{proof} According to Proposition \ref{prop-1}, $\CP_{(k)}$ is a smooth Riemannian manifold for the condition
structure. 

Let $\gamma(t)$, $0 \le t \le T,$ be a minimizing condition geodesic with endpoints $\Sigma_1$ and $\Sigma_2 \in
\CD_{(k)}$. 
Let $\gamma(t)=U_t\Sigma_tV_t^*$ be a singular value decomposition of $\gamma(t)$, \rev{choosen as in Lemma~\ref{lem-svd}}{ma11}.
Let $\si_u(t)$ be the smallest singular value of $\gamma(t)$. 
It suffices to see that $L_\ka(\Si) \le L_\ka(\ga)$ that is
$$\int_0^T \|\dot{\Sigma}_t\|_F \si_u(t)^{-1} dt \le \int_0^T \|\dot{\ga}_t\|_F \si_u(t)^{-1} dt.$$
Since
$$\dot{\ga}_t = \dot{U}_t\Sigma_tV_t^* + U_t\dot{\Sigma}_tV_t^* + U_t\Sigma_t\dot{V}_t^*,$$
with $\dot{U}_t = U_tA_t$, \secrev{$\dot{V}_t = V_tB_t$}{C21}, $A_t$ and $B_t$ skew-symmetric, we see that
$$ \|\dot{\ga}_t\|_F^2 =  
\| {A}_t\Sigma_t + \dot{\Sigma}_t - \Sigma_t {B}_t \|_F ^2 =
\| \dot{\Sigma}_t \|_F^2 +
 \| {A}_t\Sigma_t - \Sigma_t {B}_t \|_F^2
\ge \| \dot{\Sigma}_t \|_F^2
$$
because the diagonal terms in $\dot{\Sigma}_t$ are real numbers and those of 
${A}_t\Sigma_t - \Sigma_t {B}_t$ are purely imaginary \rev{when
$\K=\C$ and vanish when $\K=\R$.}{mi27} 
\secrev{When $\gamma_t$ does not belong to $\CD_{(k)}$, then the
inequality above is strict.}{C-I-4}

The second assertion is an easy consequence of the first one\thirdrev{.}{Remark 27} The third assertion is another classical characterization
of totally geodesic submanifolds, see \cite{one} Chapter 4, 
\rev{Proposition 13 or Theorem 5.}{mi28}

Finally, for log-convexity of $\al(X) = \si_u(X)^{-2}$, using \cite{BDMS} Proposition 3, it suffices to see that for
$\Sigma\in\Dk $ 
and $\dot{\Sigma}\in T_\Sigma\Dk $,
\begin{equation}\label{eq:pot}
2\|\dot{\Sigma}\|^2\|D\sigma_u(\Sigma)\|^2\geq \rev{\hess{\sigma_u^2(\Sigma)}(\dot{\Sigma},\dot{\Sigma})}{ma10},
\end{equation}
where the second derivative is computed in the Frobenius metric structure. 
\rev{
Now,
\[
D(\sigma_u)(\Sigma)(\dot \Sigma) = \sigma_u(\dot \Sigma) .
\]
Thus
\[
\left\|
D(\sigma_u)(\Sigma)(\dot \Sigma)\right\|^2
\]
is maximized for the `unit vector' (in block representation)
\[
\dot \Sigma =
\frac{1}{\sqrt{k_u}} 
\left[
\begin{matrix}
0_{k_1} & & &\\
& \cdots & &\\
& & 0_{k_{u-1}} & \\
& & & I_{k_u}
\end{matrix}
\right] .
\]
We deduce that
\[
\|D\sigma_u(\Sigma)\|^2=\frac{1}{k_u},
\]
hence
\[
2\|\dot{\Sigma}\|^2\|D\sigma_u(\Sigma)\|^2
=
\frac{ 2\|\dot{\Sigma}\|^2} {k_u} \ge 2 \sigma_u(\Sigma).
\]
The right-hand-side of \eqref{eq:pot} is
precisely
\[
\hess{\sigma_u^2(\Sigma)}(\dot{\Sigma},\dot{\Sigma})
=
D(2\sigma_u(\Sigma)\sigma_u(\dot \Sigma)) (\dot{\Sigma})
=
2\sigma_u(\dot{\Sigma})^2,
\]
and equation (\ref{eq:pot}) follows.}{ma10-mi29}\index{ma10}\index{mi29}
\end{proof}

\begin{proposition}\label{prop-2}
The map $\al = \sigma_u^{-2}$ is self-convex in $\CP_{(k)}$.
\end{proposition}

\begin{proof} By unitary invariance, we may choose as initial point a matrix $\Sigma\in\Dk $ with ordered distinct
diagonal entries 
$\sigma_1 > \ldots > \sigma_u > 0.$
We use Corollary \ref{cor-1}, with the group $G$ being  $\U_n\times\U_m$ and the action
\[
\begin{matrix}
\U_n\times\U_m\times\Pk&\longrightarrow&\Pk\\
((U,V),A)&\mapsto&UAV^*.
\end{matrix}
\]
The Lie algebra of $G$ is the set $\mathcal{A}_n \times \mathcal{A}_m$ where $\mathcal{A}_k$ is the set of $k\times k$ 
skew-symmetric matrices. 

\rev{We write $G(L)$ for the $G$-orbit of a point $L \in \CP_{(k)}$. 
In our case, 
this is the manifold of all $U L V^*$ with $U \in \U_n$, $V \in \U_m$.
The tangent space to the Lie group action at $L$ is the 
tangent manifold $T_LG(L) \subseteq T_L\CP_{(k)}$.}{mi30}

First, we note that for any $L\in\Dk $, we have
\[
(T_L G(L))^\perp=T_L\left\{ULV^*:U\in\U_n,V\in\U_m\right\}^\perp=
\]
\[
\left\{B_1L+LB_2^*:(B_1,B_2)\in\mathcal{A}_n\times\mathcal{A}_m\right\}^\perp.
\]
Let us denote by $S$ this last set. We claim that $S=\Dk$. Indeed, $\Dk\subseteq S$, because the diagonal of 
any matrix of the form $B_1L+LB_2^*$ is purely imaginary and hence orthogonal to $\Dk$. The other inclusion 
is easily checked by a dimensional argument: The dimension of $\Dk$ is $u$ and the dimension of $S$ is
\[
\dim(\Pk)-\dim\left\{B_1L+LB_2^*:(B_1,B_2)\in\mathcal{A}_n\times\mathcal{A}_m\right\},
\]
that is $\dim(\Pk)$ minus the dimension of the orbit of $L$ under the action of $\U_n\times\U_m$. We have computed 
these two quantities in Proposition \ref{prop-1}, and we immediately conclude that $\dim(S)=u$, for both $\K=\C$ and
$\K=\R$. Thus,
for all $L\in\Dk$,
\[
(T_L G(L))^\perp=\Dk.
\]
We now check the three conditions of Corollary \ref{cor-1}.
\begin{itemize}
\item $\hessk{\log(\al)}(\Sigma)$ is positive semi-definite in $(T_\Si G(\Si))^\perp$: let 
$$\dot{\Sigma}\in (T_\Si G(\Si))^\perp = T_\Si \Dk ,$$
and let $\gamma$ be a condition geodesic in $\Dk $ such that $\gamma(0)=\Sigma$, $\dot{\gamma}(0)=\dot{\Sigma}$. We have
to check that
\[
\frac{d^2}{dt^2}\log\al(\gamma(t))\left.\right|_{t=0} \geq 0.
\]
This is true as $\al$ is log-convex in $\Dk $ from Lemma \ref{lem-4}.
\item We have to check that for small enough $t$, and for
\[
\dot\Si\in (T_\Si G(\Si))^\perp= T_\Si \Dk,
\]
$D\phi_t(D)\dot\Si$ belongs to 
\[
T_{\phi_t(\Sigma)}G(\Si)^\perp= T_\Si \Dk,
\]
where $\phi_t$ is the flow of ${\rm grad}_\kappa\al$. 
In our case, $\phi_t$ can be computed exactly. Indeed,
\[
{\rm grad}_\kappa \al=\frac{1}{\al}\grad\al=-\frac{2}{k_u\sigma_u}E,
\]
where
\[
E=\diag(0,\ldots,0,\overset{k_u}{\overbrace{1,\ldots,1}}).
\]
Thus, $\grad\al$ preserves the diagonal form, and $\phi_t(\Sigma)\in\Dk$ is a diagonal matrix, for every $t$ while
defined. Thus, $D\phi_t(\Sigma)(\dot{\Sigma})$ is again a diagonal matrix, for every diagonal matrix $\dot{\Sigma}$.
This proves that the second condition of Corollary \ref{cor-1} applies to our case.
\item For $(B_1,B_2)\in \mathcal{A}_n\times\mathcal{A}_m$, the vector field $K$ on $\GL$ generated by $(B_1,B_2)$ is
\[
K(A)=\frac{d}{dt}\left(e^{tB_1}Ae^{tB_2^*}\right)\left.\right|_{t=0}=B_1A+AB_2^*.
\]
Note that
\begin{enumerate}
\item $K^*$ as a linear operator on $\GL$ satisfies $K^*(A)=B_1^*A+AB_2$.
\item $\|K(\Sigma)\|^2=\|B_1\Sigma+\Sigma B_2^*\|^2$,
\item For $w\in T_\Sigma\Pk$, $D(\|K\|^2)(\Sigma)w=2\Re\langle K^*K(\Sigma),w\rangle=2\Re\langle B_1B_1^*\Sigma+\Sigma
B_2B_2^*-2B_1\Sigma B_2^*,w\rangle$.
\item $\grad\al(\Sigma)=-\frac{2}{k_u\sigma_u^3}E \text{ where } E=\diag(0,\ldots,0,\overset{k_u}{\overbrace{1,\ldots,1}})$.
\end{enumerate}
Thus,
\[
\al(\Sigma)D(\|K\|^2)(\Sigma)(\grad\al(\Sigma))+\|K(\Sigma)\|^2\|\grad\al(\Sigma)\|^2=
\]
\[
\frac{4}{k_u\sigma_u^6}\left(-\sigma_u\Re\langle B_1B_1^*\Sigma+\Sigma B_2B_2^*-2B_1\Sigma B_2^*,E
\rangle+\|B_1\Sigma-\Sigma B_2\|^2\right).
\]
Hence, it suffices to see that $J\geq0$ where
\[
J=\|B_1\Sigma-\Sigma B_2\|^2-\sigma_u\Re\langle B_1B_1^*\Sigma+\Sigma B_2B_2^*-2B_1\Sigma B_2^*, E \rangle.
\]
Expanding this expression and writing $\Sigma'=\Sigma^*-\sigma_u E^*$, we have
\[
J=\Re \left( \trace(B_1B_1^*\Sigma\Sigma')+\trace(\Sigma'\Sigma B_2B_2^*)-2\trace(B_1\Sigma B_2^*\Sigma')\right),
\]
which by Lemma \ref{lem-5} below is a non-negative quantity. The proposition follows.
\end{itemize}
\end{proof}

\begin{lemma}\label{lem-5}
Let $\Sigma = \diag (\si_1I_{k_1}, \ldots, \si_{u-1}I_{k_{u-1}}, \si_{u}I_{k_{u}})
\in \GL $ and $\Sigma' = \diag (\si_1I_{k_1}, \ldots, \si_{u-1}I_{k_{u-1}}, 0I_{k_{u}})
\in \G\L_{m,n}$. Then, for any skew-symmetric matrices $B,C$ of respective sizes $n,m$, we have:
\[
\Re \left( \trace(BB^*\Sigma\Sigma')+\trace(\Sigma'\Sigma CC^*)-2\ \trace(B\Sigma C^*\Sigma')\right)\geq0.
\]
\end{lemma}

\begin{proof}
We denote 
$$J=\Re \left( \trace(BB^*\Sigma\Sigma')+\trace(\Sigma'\Sigma CC^*)-2\ \trace(B\Sigma C^*\Sigma')\right).$$ 
Write
\[
\Sigma=\begin{pmatrix}L&0&0\\0&\sigma_uI_{k_u}&0\end{pmatrix},\;\;\;\Sigma'=\begin{pmatrix}L&0\\0&0\\0&0\end{pmatrix},
\]
and let us write $B,C$ by blocks,
\[
B=\begin{pmatrix}B_1&B_2\\-B_2^*&B_4\end{pmatrix}
,\;\;\;C=\begin{pmatrix}C_1&C_2& C_3\\-C_2^*&C_4&C_5\\-C_3^*&-C_5^*&C_6\end{pmatrix}
\]
where $B_1,C_1$ are of the size of $L$ and $B_4,C_4$ are of the size of $I_{k_u}$. Then,
\[
\begin{array}{lll}
\trace(BB^*\Sigma\Sigma')&=&\trace((B_1B_1^*+B_2B_2^*)L^2),\\
\trace(\Sigma'\Sigma CC^*)&=&\trace(L^2(C_1C_1^*+C_2C_2^*+C_3C_3^*)),\\
\trace(B\Sigma C^*\Sigma')&=&\trace(B_1L C_1^*L+\sigma_uB_2 C_2^*L).
\end{array}
\]
Thus,
\[
J\geq \Re \left( \trace((B_1B_1^*+C_1C_1^*)L^2-2B_1L C_1^*L)\right)+
\]
\[
\Re \left( \trace((B_2B_2^*+C_2C_2^*+C_3C_3^*)L^2-2\sigma_uB_2 C_2^*L)\right).
\]
We will prove that these two terms are non-negative. For the first one, note that
\[
\Re \left( \trace((B_1B_1^*+C_1C_1^*)L^2-2B_1L C_1^*L)\right)=\|B_1L-L C_1\|^2\geq0.
\]
For the second one,
we check that for every $l$, $1\leq l\leq n-k_u$ the $l$-th diagonal entry 
of the matrix $(B_2B_2^*+C_2C_2^*+C_3C_3^*)L^2-2\sigma_uB_2 C_2^*L$ 
has a positive real part. Indeed, if we denote by $v\in\K^{k_u}$ the $l$-th row of $B_2$ 
by $w\in\K^{k_u}$ the $l$-th row of $C_2$ and by $x$ the $l$-th
row of $C_3$, we have
\[
\Re\left((B_2B_2^*+C_2C_2^*+C_3C_3^*)L^2-2\sigma_uB_2 C_2^*L\right)_{l,l}=
\]
\[
\sigma_l^2\left((\|v\|^2+\|w\|^2+\|x\|^2)-2\frac{\sigma_u}{\sigma_l}\Re\langle v,w\rangle\right)\geq
\sigma_l^2 \|v-w\|^2\geq0
\]
as $\sigma_u<\sigma_l$. This finishes the proof of Lemma \ref{lem-5} and hence of Proposition \ref{prop-2}.
\end{proof}

\section{Puting pieces together}\label{sec-5}

Before stating the main result of this section we have to introduce the following machinery:

\subsection{Second symmetric derivatives}\label{sec-5.1} 

In the case of Lipschitz-Riemann structures, the mappings we want to consider are not necessarily $\mathcal C^2$ and, to study
their convexity properties, 
an approach based on \rev{the usual covariant second derivative}{ma10} is insufficient. We will use instead the second symmetric upper derivative. 

Let $U \subseteq \R^k$ be an open set and $\phi: U \rightarrow \R$ be any function. The {\bf second symmetric upper
derivative} of $\phi$ at $x \in U$ in the direction $v \in \R^k$ is
\[
\SD\phi(x;v)=\limsup_{h\mapsto0}\frac{\phi(x+hv)+\phi(x-hv)-2\phi(x)}{h^2}
\]
which is allowed to be $\pm\infty$. If $U \subseteq \R$ is an interval, we simply write $\SD\phi(x)$ for $\SD\phi(x;1)$.

It is well-known that a continuous function $\phi$ on an interval is convex if and only if $\SD \phi(x)\geq 0$ for all
$x$ (see for example \cite{tho} Theorem 5.29). There is a stronger result due to Burkill \cite{bur} Theorem 1.1 (see
also \cite{tho} Corollary 5.31) which uses a weaker hypothesis:

\begin{theorem}[Burkill]\label{the-5.1}
Let $\phi : ]a,b[ \rightarrow \R$ be a continuous function such that $\SD \phi(x) \geq 0$ for almost all $x \in ]a,b[$,
and assume that  $\SD\phi(x)>-\infty$ for $x\in]a,b[$. Then, $\phi$ is a convex function.
\end{theorem}

Theorem \ref{the-5.1}  will allow us to assemble the pieces where convexity is proven in Proposition \ref{prop-2} to
prove our main results (\rev{Theorems \ref{th-1} and \ref{the-7.8}}{mi31}). We proceed a little more generally as the result may be
of interest in other circumstances. Let $\CM$ be a $k$-dimensional $\mathcal C^2$ manifold (not necessarily having a Riemannian
structure).

\begin{definition}\label{def-5.1}
Let $\alpha:\CM\rightarrow\R$. We say that {\bf $\SD \al$ is bounded from $-\infty$} (denoted $\SD \al > -\infty$) if
for every $x \in \CM$ there is an open neighborhood $U_x\subseteq\CM$ and a coordinate chart
$\varphi_x:U_x\rightarrow\R^k$, $\varphi_x(x)=0$ such that
\[
\SD(\alpha\circ \varphi_x^{-1})(0;v)>-\infty
\]
for every $v \in \R^k.$
\end{definition}


\rev{}{mi32}
%
%

The following lemma is a consequence of Definition \ref{def-5.1}.

\begin{lemma}\label{lem-5.2}
Let $\CM$ be a \rev{$\mathcal C^2$  manifold}{mi33}, let $\alpha : \CM \rightarrow ]0, \infty[$ be a locally Lipschitz
mapping. 
\rev{Then, $\SD \al > -\infty$ if and only if,
for any function $\phi : ]0, \infty[ \rightarrow \R$ of class $\mathcal C^2$,
$\SD (\phi \circ \alpha) >
-\infty$.}{mi34} In particular,  $\SD \al > -\infty$ if and only if  $\SD (\log \circ \al) > -\infty$.
\end{lemma}

\begin{proof}
\rev{The {\em if} part is trivial (just make $\phi(t)=t$). In order to prove
the {\em only if} part, we
assume that $\SD \al > -\infty$.}{mi34} Let $x\in\CM$ and let $\varphi_x : U_x \rightarrow \R^k$ be a coordinate chart such
that $\varphi_x(x)=0$ and 
$\SD (\alpha \circ \varphi_x^{-1})(0;v) > -\infty$ for each $v \in \R^k$.
There is a sequence $h_p \rightarrow 0$ such that
$$
\lim_{p \rightarrow \infty} \frac{\alpha(\varphi_x^{-1}(h_pv)) + \alpha(\varphi_x^{-1}(-h_pv)) - 2\alpha(x)}{h_p^2} = C
> -\infty .
$$
Let us define $H_p = \alpha(\varphi_x^{-1}(h_pv))-\alpha(x)$, and similarly  $K_p =
\alpha(\varphi_x^{-1}(-h_pv))-\alpha(x)$.  By Taylor's formula we get
\[
\phi(\alpha(\varphi_x^{-1}(h_pv))) = \phi(\al(x)) + \phi'(\al(x))H_p + \phi''(\al(x))\frac{H_p^2}{2} + o(H_p^2),
\]
and similarly
\[
\phi(\alpha(\varphi_x^{-1}(-h_pv))) = \phi(\al(x)) + \phi'(\al(x))K_p + \phi''(\al(x))\frac{K_p^2}{2} + o(K_p^2),
\]
so that
$$\frac{\phi(\alpha(\varphi_x^{-1}(h_pv)))+\phi(\alpha(\varphi_x^{-1}(-h_pv)))-2\phi(\alpha(x)}{h_p^2} = $$
$$\phi'(\alpha(x))\frac{H_p + K_p}{h_p^2} + \phi''(\alpha(x))\frac{H_p^2 + K_p^2}{2h_p^2} + \frac{o(H_p^2) +
o(K_p^2)}{h_p^2}.$$
Notice that $\lim_{p \rightarrow \infty} \frac{H_p + K_p}{h_p^2} = C.$
Since $h \ra \alpha(\varphi_x^{-1}(hv))$ is Lipschitz in a neighborhood of $0$ we have, for a suitable constant $D > 0$,
$H_p^2 \le Dh_p^2$ and 
$K_p^2 \le Dh_p^2$. Thus, taking the $\limsup$ as $p \ra \infty$ gives $\SD \phi(\alpha \circ \varphi_x^{-1})(0,v) \ge C
+ D > - \infty$ and we are done. 
\end{proof}

\subsection{Projecting geodesics on submanifolds : the Euclidean case}\label{sec-5.2} 

The following technical lemma, interesting by itself, is a consequence of Lebesgue's Density Theorem.

\begin{lemma}\label{lem-5.3}
For any locally integrable function $f$ defined in $\mathbb R$ with values
in $\mathbb R^n$, \secrev{}{M1}\rev{let $x\in \mathbb R$ be a point where
$f$ is locally integrable.
This means that $F'(x) = f(x)$ where $F$ denotes an antiderivative of $f$.}{mi35}
Then
$$\lim_{\ep \ra 0} \frac{2}{\ep^2} \int_{x}^{x+\ep}(y-x)f(y)dy = f(x).$$
\end{lemma}

\begin{proof}
Notice that, by Lebesgue's differentiation theorem, an antiderivative $F$ of $f$ exists a. e. and it is absolutly
continuous. Suppose that
$F(x)=0$. Let us define 
$$h(y) = \left\lbrace  
\begin{array}{ll}
 {F(y)}/(y-x)& \mbox{ if } y \ne x,\\
  f(x) & \mbox{ if } y = x,\\
\end{array}
\right. $$
so that $h$ is a continuous function and $F(y) = (y-x)h(y)$ for any $y$. 
Integrating by parts gives 
$$\int_x^{x+\ep} (y-x)f(y) dy = \ep F(x+\ep) - \int_x^{x+\ep} F(y) dy$$
so that
$$ \frac{2}{\ep^2} \int_x^{x+\ep} (y-x)f(y) dy =  2 \frac{F(x+\ep) - F(0)}{\ep} - \frac{2}{\ep^2} \int_x^{x+\ep}
(y-x)h(y) dy.$$
Since $h$ is continuous, by the Mean Value Theorem, there exists $ \ze \in [x, x+\ep]$ such that 
$$\frac{2}{\ep^2} \int_x^{x+\ep} (y-x)h(y) dy = \frac{2h(\ze)}{\ep^2} \int_x^{x+\ep} (y-x) dy = h(\ze) \ra h(x) = f(x)$$
as $\ep \ra 0$. On the other hand 
$$\lim_{\ep \ra 0} 2 \frac{F(x+\ep) - F(x)}{\ep} = 2f(x).$$
Thus 
$$ \lim_{\ep \ra 0} \frac{2}{\ep^2} \int_{x}^{x+\ep}(y-x)f(y)dy = 2f(x) - f(x) = f(x)$$
and we are done. 
\end{proof}

Our aim is now to see how close are a geodesic in a Lipschitz-Riemannian manifold and a geodesic in a submanifold when
they have the same tangent at a given point. Let us start to study a simple case. 

Let us consider the Lipschitz-Riemann structure defined on an 
\rev{open, $k$-dimensional}{mi36} set $\Om \subset \R^k$ containing $0$ by the scalar
product $\left\langle u,v \right\rangle_x = v^T H(x) u$ (see section \ref{sec-2.3}). 
\begin{enumerate}
 \item[1.] The matrix  $H(0)$ is supposed to have the following block structure
$$H(0) = \left( 
\begin{array}{cc}
H_p(0) & 0 \\ 
0 & H_{k-p}(0)
\end{array}
\right) .$$
\end{enumerate}
\rev{}{mi38}
We also suppose that (see section \ref{sec-2.4})
\begin{enumerate}
 \item[2.] The entries $h_{ij}(x)$ of $H(x)$ are regular at $x=0$,
\end{enumerate}
The set $\Om_p = \Om \cap (\R^p \times \left\lbrace 0 \right\rbrace)$ is a submanifold in $\Om$. We suppose that  
\begin{enumerate}
 \item[3.] $H_p$ is $\mathcal C^2$ in $\Om_p$,
\end{enumerate}
so that $\Om_p$ is in fact a smooth $\mathcal C^2$ Riemannian manifold for the induced $H-$structure. 
Let us now \rev{consider}{mi39} a vector $a \in \R^p \times \left\lbrace 0 \right\rbrace$ and three parametrized curves denoted by
$x$, $x_p$, and $y$ defined in a neighborhood of $0$ in $\R$, and such that:
\begin{enumerate}
 \item[4.] $x(0) = x_p(0) = y(0) = 0,$
 \item[5.] $\dot x (0) = \dot x_p(0) = \dot y(0) = a,$
 \item[6.] $x$ is a geodesic in $\R^k$ \rev{for the $H$-structure,}{mi37}
 \item[7.] $x_p$ is its orthogonal projection onto $\R^p \times \left\lbrace 0 \right\rbrace$,
 \item[8.] $y$ is a geodesic in $\R^p \times \left\lbrace 0 \right\rbrace$ for the induced structure.
\end{enumerate}
According to Theorem \ref{the-2.1}, $x$ has regularity $\mathcal C^{1+Lip}$ so that its second derivative exists a.e. We suppose
here that 
\begin{enumerate}
 \item[9.] The second derivative $\ddot x(t)$ is defined at $t=0$, and 
$$
\frac{d}{dt}\mid_{t=0}(H(x(t)) \dot x(t)) \in  
\frac{1}{2}\sum_{i,j}\dot x_i(0)\dot x_j(0) \partial h_{ij}(x(0))
.
$$
\end{enumerate}
In this context we have:

\begin{lemma}\label{lem-5.4} Under the hypotheses $1$ to $9$ above, the curves $x_p$ and $y$  have a contact of order
$2$ at $0$: $x_p(s) = y(s) + o(s^2).$
\end{lemma}

\begin{proof} 
\rev{
By hypothesis 3,
\begin{equation}\label{difsys1}
y \in \mathcal C^{2} .
\end{equation}
Hypothesis 8 says that $y$ is a geodesic. Because geodesics are
parametrized by arc length,
\begin{equation}\label{difsys2}
\dot y^T(s) H_p(y(s)) \dot y(s) = 1.
\end{equation}
The Euler-Lagrange equation for geodesics in is now
\begin{equation}\label{difsys3}
\frac{d}{ds}(H_p(y(s)) \dot y(s)) =  \frac{1}{2}\sum_{i,j}\dot y_i(s)\dot y_j(s) \grad h_{p,ij}(y(s)),
\end{equation}
where $h_{p,ij}$ is $h_{ij}$, seen as a function of $x_1,\ldots, x_p$.
The differential system~(\ref{difsys1}-\ref{difsys3}) 
actually defines
$y(s)$ as a curve in $\Om_p$, in function of the initial condition $(y(0),\dot y(0))$.}{mi40}

Moreover, thanks to hypothesis $9$ we have:
$$H(x(0)) \ddot x(0) + \frac{d}{dt}_{|t=0}(H(x(t))) \dot x(0) \in  \frac{1}{2}\sum_{i,j=1}^k \dot x_i(0)\dot x_j(0)
\partial h_{ij}(x(0)) .$$
When we project it onto $\R^p$ we get, with $x(t) = 
\left( 
\begin{array}{l}
x_p(t) \\ 
x_{k-p}(t)
       \end{array}
\right) ,$
$$H_p(0)  \ddot x_p(0) + \frac{d}{dt}_{|t=0}(H_p(x_p(t))) \dot x_p(0) \in  \frac{1}{2}\sum_{i,j=1}^p \dot x_{p,i}(0)\dot
x_{p,j}(0) \Pi_{\R^p}\partial h_{ij}(x(0)).$$
Since the functions $h_{ij}(x)$ for $i,j=1 \ldots p$ are regular (hypothesis 2), from Clarke \cite{cla} Proposition
2.3.15, we obtain
$$ \Pi_{\R^p}\partial h_{ij}(x(0)) =  \partial h_{p,ij}(x_p(0)) = \grad h_{p,ij}(x_p(0))$$
so that 
$$H_p(0)  \ddot x_p(0) + \frac{d}{dt}\mid_{t=0}(H_p(x_p(t))) \dot x_p(0) =  \frac{1}{2}\sum_{i,j=1}^p \dot x_{p,i}(0)\dot
x_{p,j}(0) \grad h_{p,ij}(x_p(0)),$$
Taking at $t=0$ the differential equation giving $y$ and noting that $y(0)=x_p(0)$, $\dot y(0) = \dot x_p(0)$ gives
$$\ddot y(0) = \ddot x_p(0).$$

We want to prove that $x_p(s) = y(s) + o(s^2).$ According to Taylor's formula with integral remainder, we have
$$x_p(s) - y(s) = x_p(0) - y(0) + s(\dot x_p(0) - \dot y(0)) + \int_0^s (\ddot x_p(\si) - \ddot y(\si))\si d\si,$$
so that, 
$$2\frac{x_p(s) - y(s)}{s^2} =  \frac{2}{s^2}\int_0^s (\ddot x_p(\si) - \ddot y(\si))\si d\si.$$
From Lemma \ref{lem-5.3} and hypothesis $9$, the limit of this expression exists at $s = 0$, and it is equal to $\ddot
x_p(0) - \ddot y(0) = 0$. This achieves the proof.
\end{proof}
\begin{proposition}\label{prop-projectingeuclidean}
Let $\R^k$ be endowed with the \rev{Lipschitz-Riemann}{mi09} structure defined by $\langle u,v\rangle_x=v^TH(x)u$, where the entries $h_{ij}(x)$ of $H(x)$ are regular and $H(x)$ has the block structure
$$H(x) = \left( 
\begin{array}{cc}
H_p(x) & 0 \\ 
0 & H_{k-p}(x)
\end{array}
\right),$$
for $x\in\R^k$. 
\rev{Assume that $H_p$ is $\mathcal C^2$ for all $x \in \R^p \times \{0\} \subset
\R^k$.}{ma12} Let $x:[a,b]\ra\R^k$ be a geodesic in $\R^k$
\rev{with respect to the Lipschitz-Riemann structure}{ma11}. Then, there exists a zero-measure set $Z\subseteq[a,b]$ such that for $t_0\in[a,b]\setminus Z$ the following holds:

``\rev{If $x(t_0)\in \R^p\times\{0\} \subset \R^k$ and $\dot x(t_0)\in \R^p\times\{0\} \subset \R^k$,}{ma12} then the projection $x_p(t)=\pi_{\R^p}(x(t))$ has a contact of order $2$ with $y(t)$, the unique geodesic in $\R^p$ \rev{with 
respect to the Lipschitz-Riemann structure $H_p$ with initial conditions}{ma12}
\thirdrev{$y(t_0)=\pi_{\mathbb R^p} (x(t_0))$ and $\dot y(t_0)=\pi_{\mathbb R^p} (\dot x(t_0))$}{Remark 28}''.
\end{proposition}
\begin{proof}
From Remark \ref{rem-2-1}, there exists a zero measure $Z\subseteq[a,b]$ such that for $t_0\in[a,b]\setminus Z$, $\ddot x(t_0)$ exists and
$$
\frac{d}{dt}\mid_{t=t_0}(H(x(t)) \dot x(t)) \in  \frac{1}{2}\sum_{i,j}\dot x_i(t_0)\dot x_j(t_0) \partial h_{ij}(x(t_0)).
$$
From Lemma \ref{lem-5.4}, for every such $t_0$, if in addition $x(t_0)\in \R^p\times\{0\}$ and $\dot x(t_0)\in \R^p\times\{0\}$, then $x_p(t)$ has a contact of order $2$ with $y(t)$ and we are done.
\end{proof}

\subsection{Projecting geodesics on submanifolds : the Riemannian case}\label{sec-5.3} 

Our aim, in this section, is to prove another version of Lemma \ref{lem-5.4} in a different geometric context. Let $\CM$
be a $\mathcal C^3$ Riemannian manifold with distance $d$, of dimension $k$, and let $\CN$ be a submanifold of dimension $p$. 

\begin{figure}
\centerline{\resizebox{\textwidth}{!}{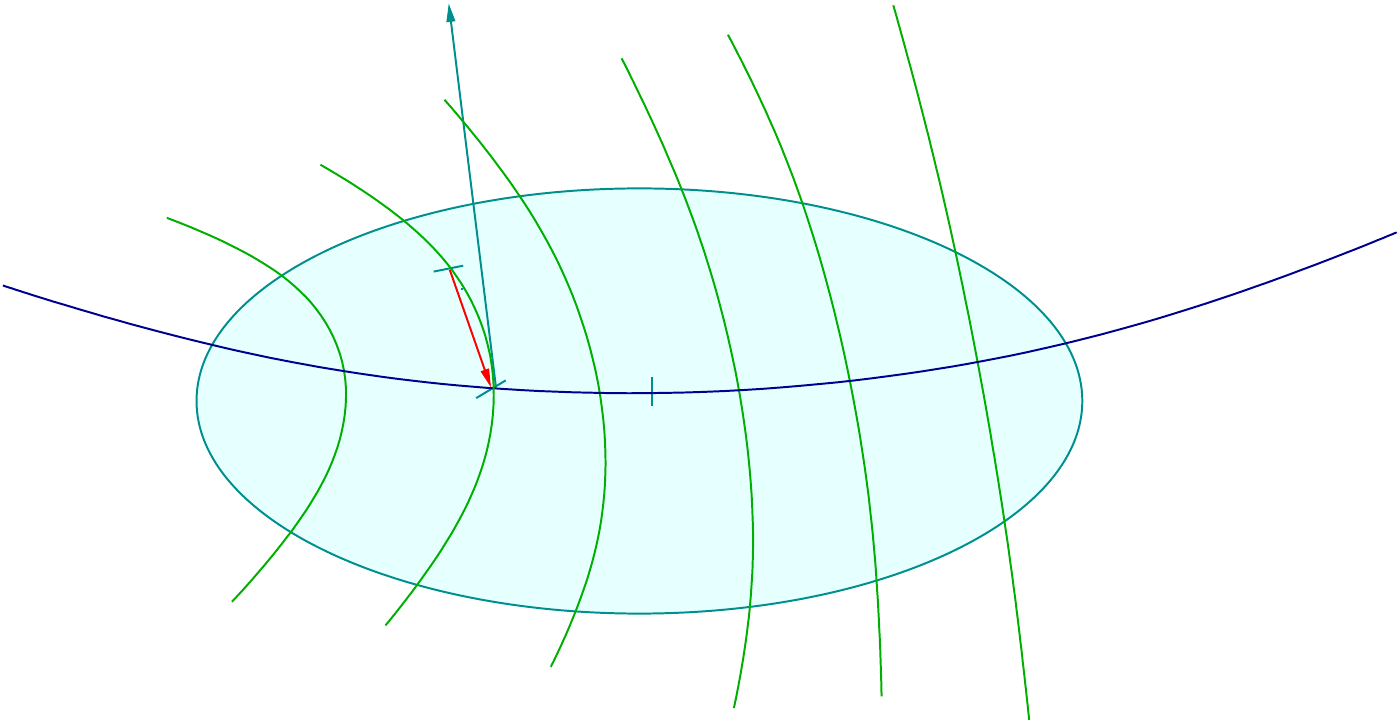}}
\caption{The projection $K: \CM \rightarrow \CN$\label{proj-K}.}
\end{figure}

Let us first define the {\bf projection onto} $\CN$ \rev{(Fig.\ref{proj-K})}{mi41}. To each $q \in \CN$ and to a vector 
$u \ne 0$ normal to $\CN$ at $q$ we associate the geodesic $\ga_{q,u}$ in $\CM$  such that $\ga_{q,u}(0) = q$ and 
$\dot\ga_{q,u}(0) = u$. Let $n \in \CN$ be given, and let $U$ be an open neighborhood of $n$ such that, for each $m \in
U$ there exists a unique geodesic arc $\ga_{q,u}(t)$, $t$ in an open interval containing $0$, contained in $U$ and
containing $m$. Thus $U$ is the union of such geodesic arcs and two of them have always a void intersection. This
picture defines a map $K : U \ra \CN$ by $K(m) = q$ if $m = \ga_{q,u}(t)$. The map $K$ is the {\bf projection map onto}
$\CN$. It has the following classical properties:
\begin{enumerate}
 \item It is defined in the neighborhood $U$ of $n \in \CN$,
 \item For each $m \in U$, $K(m)$ is the unique point in $\CM$ such that $$\inf_{q \in \CN} d(m,q) = d(m, K(m))$$
 \item $K$ is $\mathcal C^2$. 
\end{enumerate}
See Li-Nirenberg \cite{li-nir} or Beltran-Dedieu-Malajovich-Shub \cite{BDMS}.

Let $\al : \CM \ra ]0, \infty[$ be a locally Lipschitz, regular map (see section  \ref{sec-2.4}). It defines a
conformal Lipschitz-Riemann structure on 
$\CM$ associated with the inner product
$$\left\langle \cdot , \cdot \right\rangle_{\al, m} = \al(m) \left\langle \cdot, \cdot \right\rangle_{m}.$$
We call it the $\al-$structure. We suppose that $\al$ is $\mathcal C^2$ when it is restricted to $\CN$ so that $\CN$ is $\mathcal C^2$
and not only Lipschitz for the induced $\al-$structure.


\begin{proposition}\label{prop-projectingRiemannian}
Under the hypotheses above, let $\ga:[a,b]\ra\CM$ be a geodesic curve in $\CM$ for the $\al-$structure. Then, there exists a zero-measure set $Z\subseteq[a,b]$ such that for $t_0\in[a,b]\setminus Z$ the following holds:

``If $\ga(t_0)\in \CN$ and $\dot \ga(t_0)\in T_{\ga(t_0)}\CN$, then the projection \rev{$\ga_{\CN}(t)=(K\circ \ga)(t)$}{mi42} of $\ga$ onto $\CN$ has a contact of order $2$ with $\de(t)$, the unique geodesic in $\CN$ such that $\de(t_0)=\ga(t_0)$ and $\dot \de(t_0)=\dot \ga(t_0)$''.
\end{proposition}
\begin{remark}
If $\CM$, $\CN$ and $\alpha$ are assumed to be smooth then $Z=\emptyset$ in Proposition \ref{prop-projectingRiemannian}. See for example the proof of Proposition 5.9 in \cite{VV}.
\end{remark}

\begin{proof} The proof consists in a transfer from $\CM$ to $\R^k$ where we apply Proposition \ref{prop-projectingRiemannian}. Let
$$
Z=\left\{t_0\in[a,b]:\ga(t_0)\in \CN,\dot \ga(t_0)\in T_{\ga(t_0)}\CN\text{ but }\lim_{t\ra t_0}\frac{\ga_\CN(t)-\de(t)}{(t-t_0)^2}\neq0\right\}.
$$
We have to check that $Z$ is a zero measure set. It suffices to see that for every $t\in(a,b)$ there is an open interval $I$ containing $t$ and such that $I\cap Z$ has zero measure. Without 
\thirdrev{loss}{Remark 31} of generality, we may assume that $t=0$. Thus, let $t=0\in(a,b)$ and let $n=\ga(0)$.

Since $\CM$ is $\mathcal C^3$, the normal bundle to $\CN$ is $\mathcal C^2$ and there exists a $\mathcal C^2$ diffeomorphism $\phi : U \ra V
\subset \R^k$, where $V$ is an open set containing $0$, satisfying
\begin{enumerate}
 \item $\phi(n) = 0$,
 \item $\phi(U \cap \CN) = V \cap \left( \R^p \times \left\lbrace 0 \right\rbrace\right)  $,
 \item For any $q \in \CN$ and any vector $u \ne 0$ normal to $\CN$ at $q$, $\phi\left( \ga_{q,u}\right)$ is a straight
line in $\R^k$ orthogonal to
$\R^p \times \left\lbrace 0 \right\rbrace $.
\end{enumerate}
We make $\phi$ an isometry in defining on $V \subset \R^k$ a Lipschitz-Riemannian structure by
$$\left\langle D\phi(m)u, D\phi(m)v \right\rangle_{\phi(m)} = \al(m) \left\langle u,v \right\rangle_m$$
for any $m \in U$, and $u,v \in T_m\CM$. Let us denote $x = \phi(m)$, $a = D\phi(m)u$, $b = D\phi(m)v$, we also write
this scalar product
$$\left\langle a, b \right\rangle_{x} = b^T H(x) a$$
where $H$ is a locally Lipschitz map from $V$ into the $k \times k$ positive definite matrices. 

Notice that $H$ is regular because $\al$ is regular in $\CN$. 

Since for every $\hat{n}\in \CN\cap U$,
$$D\phi(\hat{n})\left( T_{\hat{n}} \CN \right)  = \R^p \times \left\lbrace 0 \right\rbrace \mbox{ and } D\phi(\hat{n})\left( \left(
T_{\hat{n}}\CN\right) ^\perp\right)  =  \left\lbrace 0 \right\rbrace \times \R^{k-p}, $$
$H(x)$ has the block structure
$$H(x) = \left( 
\begin{array}{cc}
H_p(x) & 0 \\ 
0 & H_{k-p}(x)
\end{array}
\right) .$$ 

Since $\al$ is $\mathcal C^2$ when restricted to $\CN$ we have the same regularity for the restriction of $H$ to $\R^p \times
\left\lbrace 0 \right\rbrace$.

Since $\phi$ is an isometry the curves $\phi \circ \ga$ and $\phi \circ \de$ are geodesics in $\R^k$ and $\R^p \times
\left\lbrace 0 \right\rbrace$ respectively, and, from the definition of $\phi$, the orthogonal projection (in the
Euclidean meaning) of $\phi \circ \ga$ onto $\R^p \times \left\lbrace 0 \right\rbrace$ is equal to $\phi \circ
\ga_\CN$. 

Thus, the hypotheses of Proposition \ref{prop-projectingeuclidean} are satisfied so that $\phi \circ \ga_\CN$ and $\phi \circ \de$ have an
order $2$ contact at every $t$ out of a zero measure set $Z_0$. This gives
easily an order $2$ contact for $\ga_\CN$ and $\de$ at $t\not\in Z_0$ in $\CM$ in terms of the $\al-$distance but also, since
$1/\al$ is locally Lipschitz, in terms of the initial Riemannian distance. The proposition follows.
\end{proof}

\subsection{Arriving to the main theorem}

We are now ready to state the main theorem in this section:

\begin{theorem}[Piecing together]\label{the-5.2} 
$\CM = \cup_{i=1}^\infty \CM_i$ is a $\mathcal C^3$ Riemannian manifold, enumerable union
 of the submanifolds $\CM_i$. 
Let  $\alpha : \CM \ra ]0,\infty[$ be a locally Lipschitz mapping. Assume that:
\begin{enumerate} 
\item $\alpha$ is regular,
\item For each $i$, the restriction of $\alpha$ to $\CM_i$ is $\mathcal C^2$ and self-convex in $\CM_i$,
\item $\SD \alpha > -\infty$.
\end{enumerate}
Then, $\alpha$ is self-convex in $\CM$.
\end{theorem}

\begin{proof} Once again we add to $\CM$ the $\alpha-$structure. If this theorem is false, there exists a geodesic $\ga$
in $\CM$ for the $\alpha-$structure such that 
$$\SD \log(\al(\ga(t))) < 0$$
on a positive measure set $P \subset \R$ (Theorem \ref{the-5.1} and Lemma \ref{lem-5.2}). Since an enumerable union of
zero-measure sets is also a zero-measure set, we can suppose that 
$P \subset \CM_i$
for some $i$, so that $\ga(t) \in \CM_i$ for every $t \in P$. According to the Lebesgue Density Theorem, almost all
points $t \in P$ are density points, that is
$$\lim_{\ep \ra 0} \frac{\mbox{meas}\left( P \cap [t-\ep , t + \ep ] \right) }{2\ep} = 1.$$
We remove the ``non-density points'' from $P$ to obtain a new set, also called $P$, with positive measure and only
density points. Since $\ga \in \mathcal C^{1+Lip}$ (Theorem 
\ref{the-2.1}), the second derivative $\ddot \ga(t)$ exists for almost all $t$. We also remove from $P$ the zero measure set of Proposition \ref{prop-projectingRiemannian}.

Let $t \in P$ be given. Since it is a density point of $P$, we have $s \in P$ for ``a lot of points'' close to $t$.
Since $\ga(s) \in \CM_i$ for such points,
and since $\ga$ is $\mathcal C^1$, we get
$$\dot \ga(t) \in T_{\ga(t)}\CM_i.$$
Take now the geodesic $\de$ in $\CM_i$ for the induced $\al-$structure such that $\de(t) = \ga(t)$ and $\dot \de(t) =
\dot \ga(t)$. As we have removed the zero-measure set of Proposition \ref{prop-projectingRiemannian}, $\ga_i$ and $\de$ have a contact of order $2$ at $t$. 

By self-convexity of 
$\al$ in $\CM_i$, and since $\de$ is $\mathcal C^2$ we get
$$\SD \log \circ \al \circ \de (t) = \frac{d^2}{dt^2}\log \circ \al \circ \de (t) \ge 0.$$
Let us now consider 
$$ \De^2(h) = \frac{\log \circ \al \circ \ga (t+h) + \log \circ \al \circ \ga (t-h) - 2\log \circ \al \circ \ga
(t)}{h^2}.$$
It is not difficult to prove that $t$ is a density point of
$$Q = \left\lbrace s=t+h \in P \ : \ t-h \in P \right\rbrace .$$
Let us denote by $\ga_i$ the projection of $\ga$ on $\CM_i$ (see section \ref{sec-5.3}). For the points $s=t+h \in Q$,
one has
$\ga(t+h) = \ga_i(t+h)$, $\ga(t-h) = \ga_i(t-h)$, and $\ga(t) = \ga_i(t)$, thus
$$\De^2(h) = \frac{\log \circ \al \circ \ga_i (t+h) + \log \circ \al \circ \ga_i (t-h) - 2\log \circ \al \circ \ga_i
(t)}{h^2}.$$
From the contact of order $2$ between $\ga_i$ and $\de$ we then conclude,
$$\De^2(h) = \frac{\log \circ \al \circ \de (t+h) + \log \circ \al \circ \de (t-h) - 2\log \circ \al \circ \de (t) +
o(h^2)}{h^2}.$$
Since $\de$ is $\mathcal C^2$, taking the limit as $h \ra 0$ gives
$$\lim \De^2(h) = \frac{d^2}{dt^2} \log \circ \al \circ \de (t).$$
Since this last expression is nonegative we obtain
$$\SD \log(\al(\ga(t))) \ge \lim \De^2(h) \ge 0$$
which contradicts our hypothesis $\SD \log(\al(\ga(t))) < 0$ on $P$.
\end{proof}

\section{Proof of Theorem \ref{th-1}}

Theorem \ref{th-1} is a consequence of Theorem \ref{the-5.2} applied to $\CM = \GL$ considered as the 
union of the submanifolds $\CP_{(k)}$ (see section \ref{sec-4})
and to the mapping $\alpha (A) = \si_n(A)^{-2}$, the inverse of the square of the smallest singular value of $A \in
\GL$. According to propositions \ref{prop-1} and \ref{prop-2} we just have to prove that $\al$ is a regular map and that
$\SD \al > -\infty$. Let us start with this last inequality. 

We must prove that for every $A\in\GL$, $B\in\Knm$,
\[
\SD\sigma_n^{-2}(A;B)=\limsup_{h\mapsto0}\frac{\sigma_n^{-2}(A_h)+\sigma_n^{-2}(A_{-h})-2\sigma_n^{-2}(A)}{h^2}>-\infty,
\]
where $A_h=A+hB$. Now, let $\CS_n^+$ be the set of symmetric, positive definite $n \times n$ matrices. Then,
\[
\sigma_n^{-2}(A_h)+\sigma_n^{-2}(A_{-h})=\lambda_n^{-1}(A_hA_h^*)+\lambda_n^{-1}(A_{-h}A_{-h}^*).
\]
where, $\la_n$ denotes the smallest eigenvalue. Since, for any $S \in \CS_n^+$,
$$ \lambda_n(S) = \inf_{u \in \R^n, \ \|u\|=1} u^T S u,$$
it is a concave function of $S$, and $\lambda_n^{-1}$ is convex. Thus, 
\[
\lambda_n^{-1}(A_hA_h^*)+\lambda_n^{-1}(A_{-h}A_{-h}^*)\geq
2\lambda_n^{-1}\left(\frac{A_hA_h^*+A_{-h}A_{-h}^*}{2}\right) =2\lambda_n^{-1}(AA^*+h^2BB^*).
\]
We conclude that
\[
\SD\sigma_n^{-2}(A;B)\geq\limsup_{h\mapsto0}\frac{2\lambda_n^{-1}(AA^*+h^2BB^*)-2\lambda_n^{-1}(AA^*)}{h^2}.
\]
This last quantity is bounded in absolute value \thirdrev{since}{Remark 32} $\lambda_n^{-1}$ is locally Lipschitz, so in particular
$\SD\sigma_n^{-2}(A;B)>-\infty$. 

To prove that $\al$ is regular it suffices to write it as the composition of $\mathcal C^1$ maps and of the convex
$\lambda_n^{-1}$ which is also a regular map (see 
\cite{cla} Prop. 2.3.6). This finishes the proof of our Main Theorem \ref{th-1}.

\section{The solution variety}
As in \cite{BDMS}, we are also interested in the log--convexity of $\sigma_n(A)^{-1}$ in the solution variety: 
\[
\CW=\{(A,x)\in \mathbb{GL}_{n,n+1}\times\P(\K^{n+1}):Ax=0\}.
\]

\begin{remark}
In \cite{BDMS} we have sometimes taken $A$ to lie in the unit sphere
of $\mathbb K^{n \times m}$ or even the projective space
$\mathbb P(\mathbb K^{n \times m})$. The interested reader can
check \cite{BDMS} for the relations between self-convexity in
the various settings. 
\end{remark}

\begin{theorem}\label{the-7.8}
For any condition geodesic $t \ra (A(t),x(t))$ in $\CW$, the map \rev{$t \ra \log \left(\si_n^{-2} (A(t))\right)$}{ma01} is convex. 
\end{theorem}
As we have done in the case of $\GL$, we divide the proof in several sections.
\subsection{The smooth part of $\CW$}
Let $u \leq n$ and $(k)=(k_1, \ldots , k_u) \in \N^u $ such that $k_1 + \cdots + k_u=n$. We define $\CW_{(k)}=\{(A,x)\in
\CW: A\in\CP_{(k)}\}$.

\begin{proposition}\label{prop-7.3}
For any choice of $(k)$, the set $\CW_{(k)}$ is a smooth submanifold of $\CW$, $\sigma_u$ is a smooth function and
$\alpha=\sigma_u^{-2}$ is self--convex in $\CW_{(k)}$.
\end{proposition}
\begin{proof}
Let us consider the map
\[
\begin{matrix}
\psi:&\CP_{(k)}\times\K^{n+1}\setminus \{0\}&\ra&\K^n\\
&(A,x)&\mapsto&Ax
\end{matrix}
\]
which is a smooth mapping between two smooth manifolds. 
Since $0$ is a regular value of $\psi$, its preimage $\psi^{-1}(0)$ 
is a smooth submanifold of $\CP_{(k)}\times\K^{n+1}\setminus\{0\}$. 
Moreover, $\sigma_u$ is the composition of the projection onto the 
first coordinate $\CW_{(k)}\ra\CP_{(k)}$ and the function 
$\sigma_u$ which is smooth by Proposition \ref{prop-1}.
To check that $\sigma_u$ is self--convex in $\CW_{(k)}$ we use Corollary \ref{cor-1} and proceed as in the proof 
of Proposition~\ref{prop-2}. 
Let $G=\U_n\times\U_{n+1}$, and consider the action
\[
\begin{matrix}
G\times\CW_{(k)}&\ra&\CW_{(k)}\\
((U,V),(A,x))&\mapsto&(UAV^*,Vx)
\end{matrix}
\]
Let $p=(\Sigma,e_{n+1})$ where $e_{n+1}^T=(0,\ldots,0,1)$ and $\Sigma\in\CD_{(k)}$ has ordered distinct singular values
$\sigma_1>\cdots>\sigma_u>0$. \rev{Recall that $T_pG(p)$ is the tangent
space in $p$ of the orbit $G(p)$ of $p$ by the Lie group $G$.}{mi43}
As in Propositions \ref{prop-1} and \ref{prop-2}, we have
\[
T_pG(p)=\{(B_1\Sigma+\Sigma B_2^*,B_2e_{n+1}):(B_1,B_2)\in\mathcal{A}_n\times\mathcal{A}_{n+1}\},
\]
\[
T_pG(p)^\perp=\{(\dot\Sigma,0):\dot\Sigma\in\CP_{(k)},\dot\Sigma\,\text{ is diagonal},\dot\Sigma e_{n+1}=0\}.
\]
Note that $T_pG(p)^\perp$ is isometric to the set of diagonal $n\times n$ matrices 
with eigenvalues $\sigma_1>\ldots>\sigma_u>0$ of respective multiplicities $k_1,\ldots,k_u$.

Let us check the conditions of Corollary \ref{cor-1}. By unitary invariance, we can choose a pair $p=(\Sigma,e_{n+1})$
as above.

\medskip

\noindent
{\bf 1.}  \rev{$\hessk{\log(\al)(p)}$}{ma10} is positive semi-definite in $(T_pG(p))^\perp$: let
$(\dot\Sigma,0)\in T_pG(p)^\perp$. Let $\gamma$ be a condition geodesic in $T_pG(p)^\perp$ such that
$\gamma(0)=(\Sigma,0)$, $\dot{\gamma}(0)=(\dot\Sigma,0)$. We have to check that
\[
\frac{d^2}{dt^2}\log\al(\gamma(t))\left.\right|_{t=0} \geq 0.
\]
This is true as $\al$ is log-convex in the set of diagonal $n\times n$ matrices with eigenvalues
$\sigma_1>\ldots>\sigma_u>0$ from Proposition \ref{prop-1}.

\medskip

\noindent {\bf 2.} We have to check that for small enough $t$, and for
\[
b=(\dot\Sigma,0)\in T_pG(p)^\perp,
\]
$D\phi_t(p)b$ is perpendicular to 
\[
T_{\phi_t(p)}G(\phi_t(p)),
\]
where $\phi_t$ is the flow of ${\rm grad}_\kappa\al$ in $\CW_{(k)}$. Now, as in the proof of Proposition \ref{prop-2},
the operator $\grad$ preserves the diagonal form of $(\Sigma,e_n)$ and hence $D\phi_t(p)b$ is of the form $(\Sigma',0)$
where $\Sigma'$ is diagonal with $\Sigma'e_n=0$. In particular, it is orthogonal to $T_{\phi_t(p)}G(\phi_t(p))$. Thus,
the second condition of Corollary \ref{cor-1} applies to our case.
\medskip

\noindent
{\bf 3.} For $(B_1,B_2)\in \mathcal{A}_n\times\mathcal{A}_m$, the vector field $K$ on $\CW_{(k)}$ generated by
$(B_1,B_2)$ is
\[
K(A,x)=\frac{d}{dt}\left(e^{tB_1}Ae^{tB_2^*},e^{tB_2}x\right)\left.\right|_{t=0}=(B_1A+AB_2^*,B_2x).
\]
Note that
\[
\|K(A,x)\|^2=\|B_1A+AB_2^*\|^2+\|B_2x\|^2.
\]
Thus,
\[
D(\|K\|^2)(A,x)(C,v)=\frac{d}{dt}\mid_{t=0}\left(\|K(A+tC,x+tv)\|^2\right)=
\]
\[
2\Re\langle B_1B_1^*A+AB_2B_2^*+2B_1^*AB_2^*,C\rangle+2\Re\langle B_2^*B_2x,v\rangle.
\]
Moreover,
\[
\grad\al(\Sigma,e_{n+1})=\left(-\frac{2}{k_u\sigma_u^2}E,0\right) \text{ where }
E=\diag(0,\ldots,0,\overset{k_u}{\overbrace{1,\ldots,1}}).
\]
Thus,
\[
\al(\Sigma,e_{n+1})D(\|K\|^2)(\Sigma,e_{n+1})(\grad\al(\Sigma,e_{n+1}))+\|K(\Sigma,e_{n+1})\|^2\|\grad\al(\Sigma,e_{n+1}
)\|^2=
\]
\[
\frac{2}{k_u\sigma_u^6}\left(-\sigma_u\Re\langle B_1B_1^*\Sigma+\Sigma B_2B_2^*-2B_1\Sigma B_2^*,E
\rangle+\|B_1\Sigma-\Sigma B_2\|^2+\|B_2x\|^2\right).
\]
This is positive from the proof of Proposition \ref{prop-2}.
\medskip

Hence, all the conditions of Corollary \ref{cor-1} are fulfilled and the Proposition follows.
\end{proof}

\subsection{Proof of Theorem \ref{the-7.8}}
Now we can prove Theorem \ref{the-7.8} using Theorem \ref{the-5.2} and Proposition \ref{prop-7.3}. Note that we have
$\CW=\cup_{(k)}\CW_{(k)}$ and $\alpha$ is smooth and self--convex in each $\CW_{(k)}$ by Proposition \ref{prop-7.3}.
From Theorem \ref{the-5.2} we just need to check that $\alpha$ is regular in $\CW$ and that $\SD\alpha>-\infty$. Since
\[
\alpha=\sigma_n^{-2}\circ\pi_1,
\]
where $\pi_1$ is the projection on the first coordinate, $\alpha$ is a smooth function. Now, consider the chart locally
given by $\pi_1^{-1}$, and note that $\alpha\circ\pi_1^{-1}=\sigma_n^{-2}$ is regular in $\GL$ from the proof of Theorem
\ref{th-1}. By definition, this means that $\alpha$ is regular in $\CW_{(k)}$. Using the same argument,
$\SD\sigma_n^{-2}>-\infty$ in $\GL$ also implies that $\SD\alpha>-\infty$ in $\CW$ and we are done.
\qed

\rev{}{mi44}

\section{Appendix} 

In this appendix we prove the following which gives a sufficient condition for the image 
of a submanifold under a group action to be a submanifold. 

\begin{lemma}\label{lem-7} Let $G$ be a Lie group acting on 
a smooth manifold $\CM$, and $\CD$ 
a smooth submanifold in $\CM$ 
Define on $G \times \CD$ the equivalence relation $(g,d) \CR (g',d')$ 
when $gd=g'd'$. 
Let us denote by
$$ \pi : G \times \CD \ra (G \times \CD) / \CR $$
the canonical surjection onto the quotient space, by $\immersion$ the map
$$\immersion : (G \times \CD) / \CR \ra \CM, \ \ \immersion(\pi(g,d)) = gd,$$
and by $\CP = \immersion( (G \times \CD) / \CR )$ the image of $\immersion$.
When the three following conditions are satisfied
\begin{enumerate}
\item The graph of $\CR$ is a closed \thirdrev{embedded}{Remark 33} submanifold in $(G \times \CD) \times (G \times \CD)$,
\item $\immersion$ is an immersion,
\item For every sequence $(x_k) \in (G \times \CD) / \CR$ such that $(\immersion(x_k))$ converges to 
$y \in \CP$ the sequence $(x_k)$ converges,
\end{enumerate}
then, $\CP$ is \rev{an embedded}{mi25} submanifold in $\CM$.
\end{lemma}

\begin{proof} Let $\CX$ be a manifold and let $\CR$ denote an equivalence relation defined on $\CX$. A classical
necessary and sufficient condition to define on the quotient space $\CX / \CR$ a unique quotient manifold structure
making the canonical surjection $\pi : \CX \ra \CX / \CR$ a submersion is the following:
the graph $\CG$ of the relation is an \thirdrev{embedded}{Remark 33. No need
for $\CG$ to be closed here.} submanifold in $\CX \times \CX$ and the first projection $\mathrm{pr}_1 : \CG \ra
\CX$ is a submersion (See~\cite[Th.3.5.25]{AMR}).

In the context of our lemma this condition comes from the first hypothesis and from the definition of the equivalence
relation via the group action:
\thirdrev{ 
let $((g,d),(h,e))\in \mathcal{G}$. Let $(\dot g,\dot d)\in T_{(g,d)}(\mathcal{G}\times\mathcal{D})$. Let $a(t)$ be a curve in $G$ and $b(t)$ a curve in $\mathcal{D}_{(k)}$ such that:
\[
a(0)=g,\quad \dot a(0)=\dot g,\quad b(0)=d,\quad \dot b(0)=\dot d.
\]
Then, consider the following curve contained in $(\mathcal{G}\times\mathcal{D})\times(\mathcal{G}\times\mathcal{D})$ defined by:
\[
\theta(t)=((a(t),b(t)),(a(t) g^{-1} h,h^{-1} g b(t))).
\]
It is clear that $\theta(0)=((g,d),(h,e))$ because
\[
h^{-1} g b(0) = h^{-1} g d = e . 
\]
It is also clear that
$D\mathrm{pr}_1(\theta(0))\ \theta'(0)=(\dot g,\dot d)$. Moreover, it is immediate that $\theta(t)$ is contained in $\mathcal{G}$. Thus, $\mathrm {pr}_1$ is a submersion.
}{Remark 34, proof by C., notation changes by G.}

Let $f : \CY \ra \CZ$ be a smooth map between two manifolds. Its image $f(\CY)$ is a submanifold in $\CZ$ when $f$ is an
immersion and a homeomorphism onto its image. 

\rev{By construction, $\immersion$ is smooth. 
It is a homeomorphism by the
third hypothesis and an immersion by the second one.
To check that it is injective, we have to show that
if $gd = g'd'$, then $(g,d) \CR (g',d')$. This follows
from the construction of the relation $\CR$.}{mi24} 

\end{proof}

\end{document}